\documentclass[11pt, oneside]{amsart}

\usepackage{etex}
\usepackage[T1]{fontenc}
\usepackage[utf8]{inputenc}
\usepackage[english]{babel}
\usepackage[margin=3cm]{geometry}
\usepackage[pdftex]{graphicx}

\usepackage{amsmath}
\usepackage{amsfonts}
\usepackage{amssymb}
\usepackage{amsthm}
\usepackage{booktabs}
\usepackage{paralist}
\usepackage{array}
\usepackage{xy}
\usepackage{multicol}
\usepackage{multirow}
\usepackage{fancyhdr}
\usepackage{hyperref}
\usepackage{wrapfig}
\usepackage{bm}
\usepackage{listings}
\usepackage{enumerate}
\usepackage{multirow}
\usepackage{placeins}
\usepackage{pgf}
\usepackage{tikz}
\usepackage{xcolor}
\usepackage{mathtools}
\usepackage{algorithm}
\usepackage{algpseudocode}
\usepackage{dsfont}
\usepackage{subcaption}

\newcommand{\C}{\mathbb{C}}
\newcommand{\R}{\mathbb{R}}

\DeclarePairedDelimiter{\abs}{\lvert}{\rvert}
\DeclarePairedDelimiter{\norm}{\lVert}{\rVert}	
\DeclareMathOperator{\diag}{diag}

\DeclareMathOperator{\blkdiag}{blkdiag}	
\DeclareMathOperator{\vspan}{span} 	
\DeclareMathOperator{\range}{range} 
\DeclareMathOperator{\trace}{tr} 	
\DeclareMathOperator{\dist}{d} 		

\DeclareMathOperator{\col}{col}
\DeclareMathOperator{\traceP}{\mathcal{T}}
\DeclareMathOperator{\maxdeg}{\Delta}

\newcommand{\kryl}{\mathcal{P}}		
\newcommand{\rat}{\mathcal{Q}}		
\newcommand{\poly}{\Pi}				
\newcommand{\quadform}{\psi}

\newcommand{\de}{{\normalfont\text{d}}}

\newcommand{\expec}{\mathbb{E}}	
\newcommand{\prob}{\mathbb{P}}	

\newcommand{\graph}{\mathcal{G}}

\newcommand{\one}{\mathds{1}}
\newcommand{\lap}{\mathcal{L}}
\renewcommand{\vec}[1]{\boldsymbol{#1}}		

\newcommand{\setnodes}{\mathcal{V}}
\newcommand{\edges}{\mathcal{E}}

\newcommand{\dsum}{\displaystyle\sum}
\newcommand{\dprod}{\displaystyle\prod}

\numberwithin{equation}{section}

\theoremstyle{thmstyleone}%
\newtheorem{theorem}{Theorem}[section]
\newtheorem{proposition}[theorem]{Proposition}%
\newtheorem{lemma}[theorem]{Lemma}
\newtheorem{corollary}[theorem]{Corollary}

\theoremstyle{definition}%
\newtheorem{example}[theorem]{Example}%
\newtheorem{remark}[theorem]{Remark}%

\theoremstyle{thmstylethree}%

\begin{document}
\title[Computation of the von Neumann entropy]{Computation of the von Neumann entropy of large matrices via trace estimators and rational Krylov methods}

\author{Michele Benzi}
\address{Scuola Normale Superiore, Piazza dei Cavalieri, 7, 56126 Pisa, Italy}
\curraddr{}
\email{michele.benzi@sns.it}
\thanks{}

\author{Michele Rinelli}
\address{Scuola Normale Superiore, Piazza dei Cavalieri, 7, 56126 Pisa, Italy}
\curraddr{}
\email{michele.rinelli@sns.it}
\thanks{}
\author{Igor Simunec}
\address{Scuola Normale Superiore, Piazza dei Cavalieri, 7, 56126 Pisa, Italy}
\curraddr{}
\email{igor.simunec@sns.it}
\thanks{}

\begin{abstract}
	We consider the problem of approximating the von Neumann entropy of a large, sparse, symmetric
	positive semidefinite matrix $A$, defined as $\trace(f(A))$ where $f(x)=-x\log x$. After establishing some useful 
	properties of this matrix function, we consider the use of both polynomial and rational Krylov subspace algorithms
	within two types of approximations methods, namely, randomized trace estimators and probing techniques based
	on graph colorings. We develop error bounds and heuristics which are employed in the implementation of the algorithms.
	Numerical experiments on density matrices of different types of networks illustrate the performance
	of the methods.
\end{abstract}


\maketitle

\let\thefootnote\relax\footnotetext{2010 \emph{Mathematics Subject Classification.} Primary 65F60, 15A16.}

\section{Introduction}

The von Neumann entropy \cite{vonNeumann} of a symmetric positive semidefinite matrix $A$ with 
${\trace(A) = 1}$ is defined as $S(A) = \trace(-A \log A)$, where $\log A$ is the matrix logarithm. With the usual convention that $0 \cdot \log 0 = 0$, the von Neumann entropy is given by
$$ S(A) = - \sum_{i=1}^n \lambda_i \log \lambda_i\,,$$
where $\lambda_1, \ldots , \lambda_n$ are the eigenvalues of the $n\times n$ matrix $A$.
The von Neumann entropy plays an important role in several fields including quantum statistical mechanics \cite{LL3}, 
quantum information theory \cite{BengtZycz}, and network science \cite{Braunstein}. For example, computing the von Neumann entropy is necessary
in order to determine the ground state of many-electron systems at finite temperature \cite{Aarons}. The von Neumann entropy of graphs is also an important tool
in the structural characterization and comparison of complex networks \cite{DeDomenico,Han}.

If the size of $A$ is large, the computation of $S(A)$ by means of explicit diagonalization can be too expensive, so it becomes necessary to resort to cheaper methods that compute approximations of the entropy. In recent years, a few papers have appeared devoted to this problem; see, e.g., \cite{Choi,Fuentes,Mantica,WihlerBessireStefanov}. These papers
investigate different approaches based on quadratic (Taylor) approximants, the global Lanczos algorithm, Gaussian quadrature and Chebyshev expansion. In general, the
problem of computing the von Neumann entropy of a large matrix is difficult because the underlying matrix function is not analytic in a neighbourhood of the spectrum when the matrix is singular; 
difficulties can also be expected when $A$ has eigenvalues close to zero, which is usually the case. In particular, polynomial approximation methods may converge slowly in 
such cases.

In this paper we propose to approximate the von Neumann entropy using either the probing approach developed in~\cite{Frommer2021} or a stochastic trace estimator~\cite{Hutchinson89,MMMW21-Hutch++,PerssonCortinovisKressner21}.  The main contributions of this work are outlined in the following.

In Section~\ref{subsec:polynomial-approximation} we obtain an integral expression for the entropy function $f(x) = -x \log x$ that relates it with a Cauchy-Stieltjes function~\cite[Chapter~VIII]{Widder}, and we use it to derive error bounds for the polynomial approximation of $f$, which in turn lead to a priori bounds for the approximation of $S(A)$ with probing methods. 
In order to also have a practical stopping criterion alongside the theoretical bounds, in Section~\ref{subsec:probing-method-implementation} we propose some heuristics to estimate the error of probing methods, and in Section~\ref{subsec:experiments--probing-bound-vs-estimate} we demonstrate their reliability with numerical experiments.
In Section~\ref{subsec:probing-methods} we also use properties of symmetric $M$-matrices to show that, in the case of the graph entropy, the approximation obtained with a probing method is a lower bound for the exact entropy.

Both probing methods and stochastic trace estimators require the computation of a large number of quadratic forms with $f(A)$, which can be efficiently approximated using Krylov subspace methods~\cite{FrommerSimoncini,Guettel13}. 
In Section~\ref{subsec:rational-krylov-poles} we propose to combine polynomial Krylov iterations with rational Krylov iterations that use asymptotically optimal poles for Cauchy-Stieltjes functions~\cite{MasseiRobol21}. 
In Section~\ref{sec:a-posteriori-error-bound} we obtain new a posteriori error bounds and estimates for this task, building on the ones presented in~\cite{GuettelThesis}, and we discuss methods to compute them efficiently. 
For the case of the graph entropy, we make use of a desingularization technique introduced in~\cite{BenziSimunec} that exploits properties of the graph Laplacian to compute quadratic forms more efficiently.

The resulting algorithm can be seen as a black box method that only requires an input tolerance $\epsilon$ and 
computes an approximation of $S(A)$ with relative accuracy $\epsilon$. While this accuracy is not guaranteed when the rigorous bounds are replaced by heuristics, we found 
the algorithm to be quite reliable in practice.

Our implementation of the probing algorithm is compared to a state-of-the-art randomized trace estimator developed in~\cite{PerssonCortinovisKressner21} with several numerical experiments in Section~\ref{sec:numerical-experiments}, in which we approximate the graph entropy of various complex networks.


The rest of the paper is organized as follows. In Section~\ref{sec:properties-von-neumann-entropy} we recall some properties of the von Neumann entropy and we obtain bounds for the polynomial approximation of $f(x) = -x \log x$, which we use in the subsequent sections. In Section~\ref{sec:trace-estimation} we give an overview of probing methods and stochastic trace estimators and we derive bounds for the convergence of probing methods. In Section~\ref{sec:krylov-methods} we describe the computation of quadratic forms using Krylov subpsace methods and we derive a posteriori error bounds and estimates. In Section~\ref{sec:algorithm-overview} we summarize the overall algorithm and we discuss heuristics and stopping criteria, and in Section~\ref{sec:numerical-experiments} we test the performance of the methods on density matrices of several complex networks. Finally, Section~\ref{sec:conclusions} contains some concluding remarks.

\section{Properties of the von Neumann entropy}
\label{sec:properties-von-neumann-entropy}

A \emph{density matrix} $\rho$ is a self-adjoint, positive semi-definite linear operator with unit trace acting on a complex Hilbert space.
In quantum mechanics, density operators describe states of quantum mechanical systems. Here we consider only finite dimensional
Hilbert spaces, so $\rho$ is an $n\times n$ Hermitian matrix. For simplicity, here we focus on real symmetric matrices; all the results are easily extended to the complex case.

Recall that the \emph{von Neumann entropy} of a system described by the density matrix $\rho$ \cite{vonNeumann} is given by
\begin{equation}
	\label{eqn:entropy_definition_eigenvalues}
	S(\rho)=-\sum_{\lambda\in\sigma(\rho)}\lambda\log(\lambda)=-\trace(\rho\log \rho),
\end{equation}
where $\sigma(\rho)$ denotes the spectrum of $\rho$,
under the convention that $0\cdot \log(0)=0$. 
Here $\log(x)$ is the natural logarithm. Note that in the literature the entropy is sometimes defined using $\log_2(x)$ instead, but since $\log_2(x)=\log(x)/\log(2)$ the two definitions are equivalent up to a scaling factor. We also mention that in the original definition the entropy includes the factor
$\kappa_B$ (the Boltzmann constant), which we omit.

A straightforward way to compute $S(\rho)$ is by diagonalization. However, this approach is unfeasible when the dimension is very large. Here we propose some approaches to compute approximations to the von Neumann entropy based on the trace estimation of matrix functions. 

Note that the formula \eqref{eqn:entropy_definition_eigenvalues} is well defined even without the hypothesis of unit trace. Moreover, if $\rho$ is given in the form $\rho=\gamma A$ with $\gamma>0$, we have the relation
\begin{equation}
	\label{eqn:entropy_normalization}
	S(\rho)=-\gamma \trace(A\log(A))-\gamma \log(\gamma)\trace(A)=\gamma S(A)-\gamma \log(\gamma)\trace(A).
\end{equation}

In quantum mechanics, the von Neumann entropy of a density matrix gives a measure of how much a density matrix is distant from being a \emph{pure state}, that is a rank-1 matrix with only one nonzero eigenvalue equal to one \cite{Wehrl}. For any density matrix of size $n$, we have $0\leq S(\rho)\leq \log(n)$. Moreover, $S(\rho)=0$ if and only if $\rho$ is a pure state and $S(\rho)=\log(n)$ if and only if $\rho=\frac{1}{n}I$ \cite{vonNeumann,Wehrl}.

The von Neumann entropy also has applications in network theory. Recall that a graph $\graph$ is given by a set of nodes $\setnodes=\{ v_1,\dots,v_n \}$ and a set of edges $\edges\subset \setnodes \times \setnodes$. The graph is called \emph{undirected} if $(v_i,v_j)\in\edges$ if and only if $(v_j,v_i)\in\edges$ for all $i\neq j$. A \emph{walk} on the graph of length $\ell$ is a sequence of $\ell+1$ nodes where two consecutive nodes are linked by an edge, for a total of $\ell$ edges. We say that $\graph$ is \emph{connected} if there exists a directed path from node $i$ to
node $j$ for any pair of nodes $(i, j)$.  The \emph{adjacency matrix} of a graph is a matrix $A\in\R^{n\times n}$ such that $[A]_{ij}=1$ if $(v_i,v_j)\in \edges$, and $0$ otherwise. The \emph{degree} of a node $\deg(v_i)$ is the number of nodes that are linked with $v_i$ by an edge. It can be expressed in the form $\deg(v_i)=[A\vec{1}]_i$, where $\vec{1}$ is the vector of all ones. Finally, the \emph{graph Laplacian} is given by 
\begin{equation}
	\label{eqn:graph_laplacian}
	\lap = D-A,\quad D = \diag((\deg v_i)_{i=1}^n)=\diag(A\vec{1}).
\end{equation}
Note that $\lap$ is a singular positive semidefinite matrix and the eigenvalue $0$ is simple if and only if $\graph$ is connected, in which case the associated one-dimensional eigenspace is spanned by $\vec{1}$.
Given the density matrix $\rho = \lap/\trace(\lap)$, the von Neumann entropy of $\graph$ is defined as $S(\rho)$ \cite{Braunstein,Choi}. 

It should be mentioned that, strictly speaking, the graph entropy defined in this manner is not a ``true'' entropy, since it does not satisfy the sub-additivity requirement. In other words, the graph entropy could decrease when an edge is added to the graph, see \cite{DD1}.   For this reason, different notions of graph entropy have been proposed in the literature; see, e.g., \cite{DD2,DD3}.  These entropies still have the form of a von Neumann entropy, $S = -\trace(\rho \log \rho)$, but with a different definition of the density matrix $\rho$ which, however, is still expressed as a function of a matrix  (Hamiltonian) associated to the graph. Therefore, the techniques developed in this paper are still applicable, at least in principle.

\subsection{Polynomial approximation}
\label{subsec:polynomial-approximation}

Denote the set of all polynomials with degree at most $k$ with $\poly_k$. 
For a continuous function $f:[a,b] \to \R$, define the error of the best
uniform polynomial approximation in $\poly_k$ as 
\begin{equation}
	E_k(f,[a,b])=\min_{p\in\poly_k}\max_{x\in[a, b]}\lvert f(x)-p(x)\rvert.
\end{equation}
If $A$ is a symmetric (or Hermitian) matrix with $\sigma(A)\subset[a,b]$, estimating or bounding this quantity is crucial in the study of polynomial approximations of the matrix function $f(A)$ \cite{BeckermannReichel09} and for establishing decay bounds for the entries of $f(A)$ \cite{BenziGolub,DemkoMossSmith,Frommer2018}.

An important case is the inverse function $f(x)=1/x$ for $x\in[a,b]$, $0<a<b$, for which we have 
\begin{equation}
	\label{eqn:inverse_polyapprox}
	E_k(1/x,[a,b])=\frac{(\sqrt{\kappa}+1)^2}{2b}\left( \frac{\sqrt{\kappa}-1}{\sqrt{\kappa}+1}\right)^{k+1},
\end{equation}
where
$\kappa=b/a$; see, e.g., \cite{Meinardus}.

In general, an inequality of the form  
\begin{equation}
	\label{eqn:Bernstein_polyapprox}
	E_k(f,[a,b])\leq Cq^k,	
\end{equation}
is equivalent to the fact that $f$ is analytic over an ellipse containing $[a,b]$ 
for some computable parameters $C>0$ and $0<q<1$; see \cite[Theorem 73]{Meinardus} for more details.

The function $f(x)=x\log(x)$ is not analytic on any neighborhood of $0$, so we cannot expect to find a geometric decay as in \eqref{eqn:Bernstein_polyapprox} if we consider $[0,b]$, $b>0$, as the interval of definition. However, since $f$ is 
continuous, by the Weierstrass Approximation Theorem $E_k(f,[0,b])$ must go to $0$ as $k\to\infty$. A more precise estimate is the following \cite{WihlerBessireStefanov}, derived by computing the coefficients of the Chebyshev expansion of $f(x)$:
\begin{equation}
	\label{eqn:entropy_polyapprox_Chebyshev}
	E_k(f,[0,b]) \leq \frac{b}{2k(k+1)}\quad \text{for all } k\geq 1.
\end{equation}
This shows that the decay rate of the error is algebraic in $k$. Our approach is based on an integral representation and leads to sharper bounds.

Recall that a Cauchy-Stieltjes  function \cite{BenziSimoncini,MasseiRobol21} has the form
\begin{equation}
	\label{eqn:Cauchy-Stieltjes}
	f(z)=\int_0^\infty \frac{\de \mu(s)}{s+z},\quad z\in \C\setminus [0,+\infty),
\end{equation}
where $\mu$ is a (possibly signed) real measure supported on $[0,+\infty)$ and the integral is absolutely convergent. An example is given by the function $\log(1 + z)/z$ that has the expression
\begin{equation*}
	\frac{\log(1 + z)}{z} = \int_1^\infty \frac{1}{s(s + z)} \de s.
\end{equation*}
It is easy to check that the above identity also holds for $z \in (-1, 0)$. With the change of variable $x=1+z$ and some simple rearrangements, we get the following integral expression for the entropy:
\begin{equation}
	\label{eqn:entropy-integral-expression}
	- x \log(x) = \int_1^\infty \frac{x(1-x)}{s(s + x - 1)} \de s, \qquad x \in [0, 1].
\end{equation}
With the additional change of variable $s = t+1$, we can rewrite~\eqref{eqn:entropy-integral-expression} in the form
\begin{equation}
	\label{eqn:entropy-integral-expression-cs}
	-x \log(x) = x(1-x) \int_0^\infty \frac{1}{(t + x)(t + 1)} \de t, \qquad x \in [0, 1].
\end{equation}
Note that the above identities also hold for $x = 0$ and $x = 1$, because of the factor $x(1-x)$ in front of the integral.
This shows that although the entropy function $-x\log(x)$ is not itself a Cauchy-Stieltjes function, we can recognize a factor with the form \eqref{eqn:Cauchy-Stieltjes} in its integral representation. This observation will be important later for the selection of poles in a rational Krylov method; see Section~\ref{subsec:rational-krylov-poles}.

Integral representations have proved useful for many problems related to polynomial approximations and matrix functions; see e.g. \cite{BenziRinelli,BenziSimoncini,Frommer2018}. The following result shows that we can derive  explicit bounds for $E_k(f,[a,b])$. 
\begin{lemma} \label{lemma:polyapprox_int_transform}
	Let $f(x)$ be defined for $x\in[a,b]$ by 
	\begin{equation}
		\label{eqn:integraltransform_f_generic}
		f(x)=\int_0^{\infty}g_t(x)\, \de t,
	\end{equation}
	where $g_t(x)$ is continuous for $(t,x)\in (0, \infty) \times [a, b]$ and the integral \eqref{eqn:integraltransform_f_generic} is absolutely convergent.
	Then
	\begin{equation}
		\label{eqn:error_bound_integraltransform}
		E_k(f,[a,b])\leq \int_0^\infty E_k(g_t(x),[a,b])\,\de t,
	\end{equation}
	for all $k\geq 0$, provided that the integral on the right-hand side is finite.
\end{lemma}
\begin{proof}
	For all $t \in (0, \infty)$ and for any degree $k \ge 0$, since $g_t(x)$ is continuous over the compact interval~$[a,b]$ there exists a unique polynomial $p_k^{(t)}(x)\in \poly_k$  such that
	\begin{equation*}
		E_k(g_t(x),[a,b])=\max_{x\in[a,b]}\abs{g_t(x)-p_k^{(t)}(x)}.
	\end{equation*}
	For all $x\in [a,b]$ we can define 
	\begin{equation*}
		p_k(x) := \int_0^{\infty}p_k^{(t)}(x)\,\de t.
	\end{equation*}
	This is well defined: for all $x \in [a, b]$, the mapping $t\mapsto p^{(t)}_k(x)$ is continuous for all~$t\in (0,\infty)$ in view of 
	\cite[Theorem 24]{Meinardus}, since $g_t(x)$ is uniformly continuous for $(t,x)$ in the compact sets contained in $(0,\infty)\times [a,b]$. Moreover, the integral is finite since
	\begin{align*}
		\int_0^\infty \abs{p_k^{(t)}(x)} \de t &\le \int_0^\infty \abs{g_t(x)} \de t + \int_0^\infty \abs{g_t(x) - p_k^{(t)}(x)} \de t \\
		&\le \int_0^\infty \abs{g_t(x)} \de t + \int_0^\infty E_k(g_t(x), [a, b]) \de t < +\infty.
	\end{align*}
	We want to show that $p_k(x)$ is also a polynomial.
	Consider the expression
	\begin{equation*}
		p_k^{(t)}(x)=\sum_{i=0}^ka_i(t)\,x^i,
	\end{equation*}
	which defines the coefficients $a_i(t)$ as functions of $t$. Formally, we have
	\begin{equation*}
		p_k(x)=\sum_{i=0}^k\left( \int_0^\infty a_i(t)\,\de t \right)\cdot x^i =\sum_{i=0}^k\bar{a}_ix^i,
	\end{equation*} 
	where $\bar{a}_i:=\int_0^\infty a_i(t)\,\de t$. To conclude, we need to show that this expression is well defined, that is, the functions $a_i(t)$ are integrable in $t$. 
	Consider $k+1$ distinct points $x_0,\dots,x_k$ in the interval $[a,b]$ and let $\vec a(t)=[a_0(t),\dots,a_k(t)]^T$, $\vec p(t)=[p_k^{(t)}(x_0),\dots,p_k^{(t)}(x_k)]^T$. We have that $V \vec a(t)=\vec p(t)$, where 
	\begin{equation*}
		V=\begin{bmatrix}
			1&x_0&x_0^2&\cdots&x_0^k\\
			1&x_1&x_1^2&\cdots&x_1^k\\
			\vdots&\vdots &\vdots &\ddots&\vdots\\
			1&x_k&x_k^2&\cdots&x_k^k
		\end{bmatrix}
	\end{equation*}
	is a Vandermonde matrix. Since $V$ is nonsingular, we have that $\vec a(t)=V^{-1}\vec p(t)$. Then, if we let $V^{-1}=(c_{ij})_{i,j=0}^k$, 
	we obtain the expression 
	\begin{equation*}
		a_i(t)=\sum_{j=0}^kc_{ij}p_k^{(t)}(x_j),\quad i=0,\dots,k,
	\end{equation*}
	where $c_{ij}$ is independent of $t$ for all $i,j$. This shows that, for all $i$, $a_i(t)$ is continuous for $t \in (0,\infty)$, and we have
	\begin{equation*}
		\abs{\bar{a}_i} \le \int_0^\infty \abs{a_i(t)}\, \de t\leq \sum_{j=0}^k \abs{c_{ij}} \int_0^\infty \abs*{p_k^{(t)}(x_j)} \,\de t < +\infty,\quad i=0,\dots,k, 
	\end{equation*}
	hence $\bar{a}_i$ is well defined for $i=0,\dots,k$ and $p_k(x)$ is a polynomial of degree at most~$k$.
	Finally, we have
	\begin{align*}
		E_k(f,[a,b])&\leq \max_{x\in[a,b]}\abs{f(x)-p_k(x)}\\
		&\leq \max_{x\in[a,b]}\int_0^\infty \abs{g_t(x)-p_k^{(t)}(x)}\,\de t\\
		&\leq \int_0^\infty E_k(g_t(x),[a,b])\,\de t.
	\end{align*}
	This concludes the proof.
\end{proof}

The following result gives a new bound for $E_k(x\log x,[a,b])$.

\begin{theorem} \label{thm:entropy_polyapprox}
	Let $0\leq a< b$ and $\gamma=a/b$. We have 
	\begin{equation}
		E_k(x\log(x),[a,b])
		\leq b\,(1-\sqrt{\gamma})\frac{1+\gamma+2 k\sqrt{\gamma}}{4(k^2-1)}\left( \frac{1-\sqrt{\gamma}}{1+\sqrt{\gamma}} \right)^{k},
		\label{eqn:entropy_polyapprox_integral_explicit}
	\end{equation}
	for all $k\geq 2$.
\end{theorem}

\begin{proof}
	Let $f(x)=x\log(x)$. Notice that $f(x)=bf(b^{-1}x)+\log(b)x$,
	so 
	\begin{equation*}
		E_k(f(x),[a,b])=E_k(bf(b^{-1}x),[a,b])=b\,E_k(f(x),[\gamma,1])
	\end{equation*}
	since $k\geq 2$ and we can ignore terms of degree $1$ for the polynomial approximation.
	For $x\in[\gamma,1]$ we can use the representation \eqref{eqn:entropy-integral-expression-cs}. Let $g_t(x):=\frac{x(1-x)}{(1+t)(x+t)}$ be the integrand in \eqref{eqn:entropy-integral-expression-cs} for all $t>0$. Note that $g_t(x)$ is a continuous fuction in the variables $(t, x) \in (0, \infty) \times [a, b]$, so we can apply Lemma \ref{lemma:polyapprox_int_transform}. Then can we write $g_t(x)$ as 
	\begin{equation}
		\label{eqn:entropy_polyapprox_integrand_decomposition}
		g_t(x)= \frac{x}{x+t}-\frac{x}{1+t} =1-\frac{t}{x+t}-\frac{x}{1+t},
	\end{equation}
	and since $k\geq 2$ and $1-\frac{x}{1+t}$ has degree $1$, we get that 
	\begin{equation*}
		E_k(g_t(x),[\gamma,1])=E_k(t/(x+t),[\gamma,1])=tE_k(1/x,[\gamma+t,1+t]).
	\end{equation*}
	Hence, by using \eqref{eqn:inverse_polyapprox} and Lemma \ref{lemma:polyapprox_int_transform} we get 
	\begin{equation}
		E_k(x\log(x),[a,b])
		\leq 
		b\,\int_0^\infty \frac{t\left(\sqrt{\kappa(t)}+1\right)^2}{2(1+t)}\left( \frac{\sqrt{\kappa(t)}-1}{\sqrt{\kappa(t)}+1} \right)^{k+1}\,\de t,
	\end{equation}
	where  $\kappa(t)=(1+t)/(\gamma+t)$. In order to bound the integral, consider the identities 
	\begin{equation*}
		\frac{\sqrt{\kappa(t)}-1}{\sqrt{\kappa(t)}+1}=\frac{\left( \sqrt{1+t}-\sqrt{\gamma+t} \right)^2}{1-\gamma},\quad \left( \sqrt{\kappa(t)}+1 \right)^2=\frac{\left(\sqrt{1+t}+\sqrt{\gamma+t}\right)^2}{\gamma+t}.
	\end{equation*}
	We have
	\begin{align}
		& \int_0^\infty \frac{t\left(1+\sqrt{\kappa(t)}\right)^2}{2(1+t)}\left( \frac{\sqrt{\kappa(t)}-1}{\sqrt{\kappa(t)}+1} \right)^{k+1}\,\de t \nonumber \\
		&= 
		\frac{1}{2}\int_0^\infty \frac{t}{(\gamma+t)(1+t)} \cdot  \frac{\left(\sqrt{1+t}+\sqrt{\gamma+t}\right)^2 \left( \sqrt{1+t}-\sqrt{\gamma+t} \right)^{2k+2}}{(1-\gamma)^{k+1}}\,\de t \nonumber \\
		& \leq 
		\frac{1}{2(1-\gamma)^{k-1}}\int_0^\infty \left( \sqrt{1+t}-\sqrt{\gamma+t} \right)^{2k}\, \de t. \label{eqn:entropy_polyapprox_implicit}
	\end{align}
    By checking the derivative, it can be shown that \begin{equation*}
        F(t)=\frac{\sqrt{1+t} \, \sqrt{\gamma+t} \left(\sqrt{1+t}-\sqrt{\gamma+t}\right)^{2k}
        \left(\frac{\left(\sqrt{1+t}-\sqrt{\gamma+t}\right)^4}{2k+2}-\frac{(
        1-\gamma)^2}{2k-2}\right)}{2 \left(\sqrt{1+t} \, \sqrt{\gamma+t}-\gamma-t\right)
        \left(1+t-\sqrt{1+t} \, \sqrt{\gamma+t}\right)}
    \end{equation*}
    is an antiderivative of $\left(\sqrt{1+t}-\sqrt{\gamma+t}\right)^{2k}$. Since 
    $ \lim_{t\to\infty}F(t)=0$,
    we deduce that
    \begin{align*}
        \int_0^{\infty} \left(\sqrt{1+t}-\sqrt{\gamma+t}\right)^{2k} \de t 
        &=
        \frac{\sqrt{\gamma} \left(1-\sqrt{\gamma}\right)^{2k}
        \left( \frac{(1-\gamma)^2}{2k-2} - \frac{\left(1-\sqrt{\gamma}\right)^4}{2k+2}\right)}{2 \left(1-\sqrt{\gamma}\right) \left(\sqrt{\gamma}-\gamma\right)}
        \\
        &=
        2(1-\sqrt{\gamma})^{2k}\frac{1 + \gamma + 2 \sqrt{\gamma} \, k}{2(k^2-1)}.
    \end{align*}
    This, combined with \eqref{eqn:entropy_polyapprox_implicit}, concludes the proof.
\end{proof}

\begin{remark}
	The result of Theorem \ref{thm:entropy_polyapprox} is an improvement of \eqref{eqn:entropy_polyapprox_Chebyshev}. When $a>0$, the right-hand side in \eqref{eqn:entropy_polyapprox_integral_explicit} is the product of an algebraic and a geometric factor, so the decay is asymptotically faster. When $a=0$, we have $\gamma=0$ and \eqref{eqn:entropy_polyapprox_integral_explicit} becomes
	\begin{equation*}
		E_k(x\log(x),[0,b])\leq \frac{b}{4(k^2-1)}, 
	\end{equation*}
	which is better than \eqref{eqn:entropy_polyapprox_Chebyshev} for all $k\geq 2$. 
	The new bound~\eqref{eqn:entropy_polyapprox_integral_explicit} is compared with \eqref{eqn:entropy_polyapprox_Chebyshev} in Figure~\ref{fig:entropy-scalar-polyapprox-bound}.
\end{remark}

\begin{figure}[htbp]
	\centering
	\includegraphics[width=0.6\textwidth]{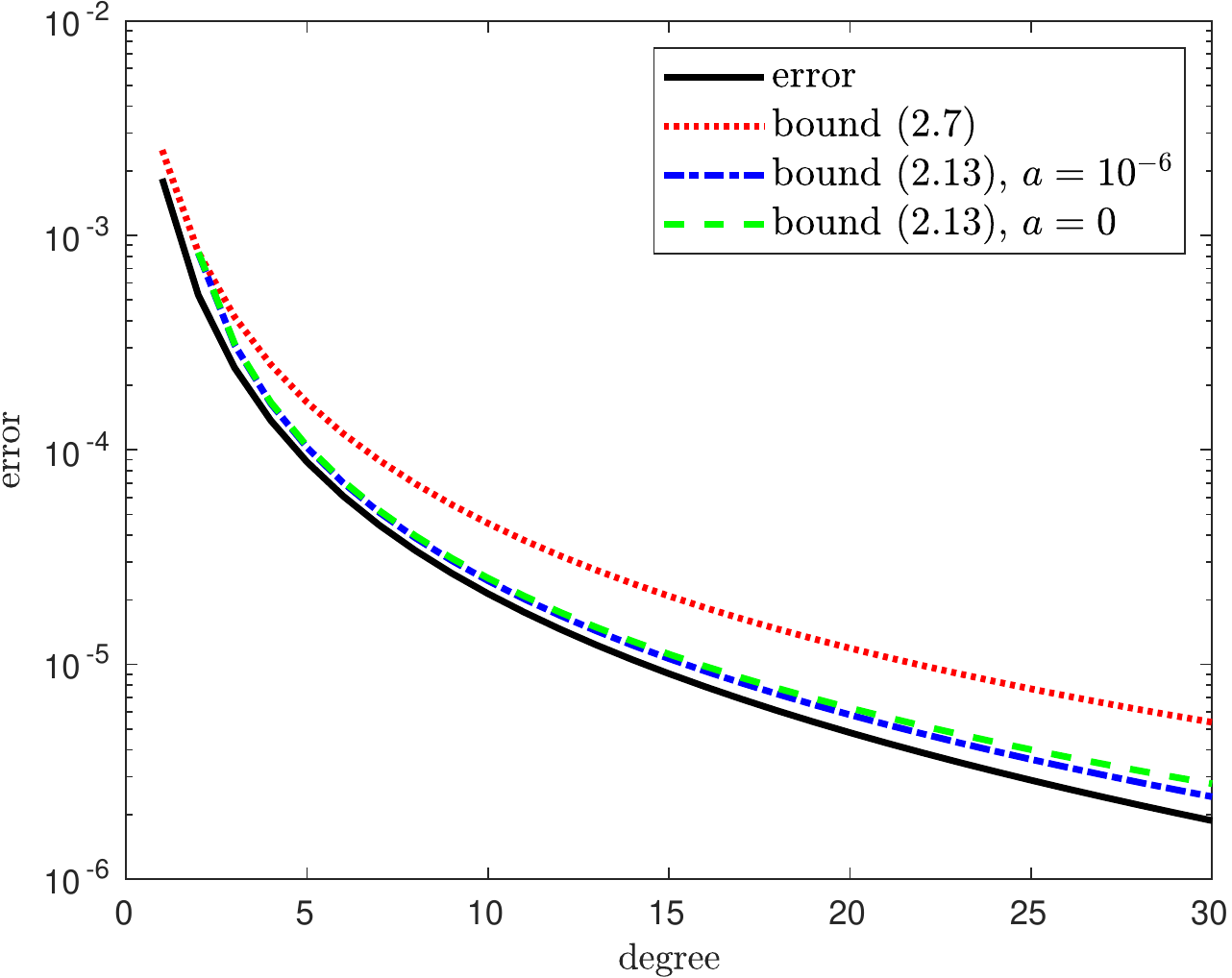}
	\caption{
	Comparison of the bounds \eqref{eqn:entropy_polyapprox_Chebyshev} and \eqref{eqn:entropy_polyapprox_integral_explicit} with the error of the polynomial approximation of the entropy function $x \log x$ on the interval $[10^{-6}, 10^{-2}]$. 	
	\label{fig:entropy-scalar-polyapprox-bound}}
\end{figure}

\section{Trace estimation}
\label{sec:trace-estimation}

Recall that a function $f(A)$ of a symmetric matrix $A\in\R^{n\times n}$ can be defined via a spectral decomposition $A=U\Lambda U^T$, where $U$ is orthogonal and $\Lambda=\diag(\lambda_1,\dots,\lambda_n)$ is the diagonal matrix containing the eigenvalues. Then
\begin{equation*}
	f(A):=Uf(\Lambda)U^T,\quad f(\Lambda):=\diag(f(\lambda_1),\dots,f(\lambda_n)),
\end{equation*} 
provided that $f$ is defined on the eigenvalues of $A$.
A matrix function can be computed via diagonalization, using polynomial or rational approximants, or with algorithms tailored to specific matrix functions, such as scaling and squaring for the matrix exponential. We direct the reader to~\cite{Higham} for more details on these methods.
In this section we present algorithms to estimate the trace of $f(A)$ that do not require the explicit computation of 
the diagonal entries of $f(A)$. It is necessary to resort to this class of algorithms whenever the size of $A$ is too large 
to make the computation of $f(A)$ (or its diagonal) feasible.
In the sections that follow, we prefer to use the notation $A$ instead of~$\rho$ for matrices, since most of our results are applicable in general for symmetric positive semidefinite matrices, and not only to density matrices.

\subsection{Probing methods}
\label{subsec:probing-methods}

Let us summarize the approach described in \cite{Frommer2021} to compute the trace of a matrix function $f(A)$, where $A\in\R^{n\times n}$ is a sparse matrix. Recall that the graph $\graph(A)$ associated with $A$ has nodes $\setnodes=\{1,\dots,n\}$ and edges $\edges=\{ (i,j):[A]_{ij}\neq 0, i\neq j \}$, and 
the \emph{geodesic distance} $\dist(i,j)$ between $i$ and $j$ is the shortest length of a walk that starts with $i$ and ends with $j$, and is $\infty$ if no such walk exists or $0$ if $i=j$.

Let $V_1,\dots,V_{s}$ be a coloring (i.e. a partitioning) of the set $\{ 1,\dots,n \}$. We write $\col(i)=\col(j)$ if and only if $i$ and $j$ belong to the same set. 
We have a \emph{distance}-$d$ \emph{coloring} if $\col(i)\neq \col(j)$ whenever $\dist(i,j)\leq d$, where $\dist(i,j)$ is the geodesic distance between $i$ and $j$ in the graph $\graph(A)$ associated with $A$.
The associated \emph{probing vectors} are
\begin{equation}
	\label{eqn:probing_vectors}
	\vec{v}_\ell=\sum_{i\in V_\ell}\vec{e}_i\in \R^n,
	\quad 
	\ell=1,\dots,s,
\end{equation} 
where $\vec{e}_i$ is the $i$-th vector of the canonical basis. The induced approximation of the trace is 
\begin{equation}
	\label{eqn:trace_probing}
	\traceP_d(f(A)):=\sum_{\ell=1}^s \vec{v}_\ell^Tf(A)\vec{v}_\ell.
\end{equation}

The problem of finding the minimum number $s$ of colors, or sets, in the partition needed to get a distance-$d$ coloring for a fixed $d$ is NP-complete for general graphs \cite{Kubale04book}. It is important to keep $s$ as small as possible since, in view of \eqref{eqn:trace_probing}, it is the number of quadratic forms needed to compute $\traceP_d(f(A))$. A greedy and efficient algorithm to get a quasi optimal distance-$d$ coloring is given by Algorithm~\ref{algorithm:greedy_coloring}~\cite[Algorithm 4.2]{SchimmelThesis}.

\begin{algorithm}
    \caption{Greedy algorithm for a distance-$d$ coloring}\label{algorithm:greedy_coloring}
    \begin{algorithmic}[1]
    \Require Graph $G=(V,E)$ with $V=\{ 1,\dots,n \}$ and distance $d$
    \Ensure Distance-$d$ coloring $\operatorname{col}$
    \State $\col(1)=1$
    \For{$i=2:n$}
    \State {$W_i=\{ j\in\{ 1,\dots,i-1 \}:\dist(i,j)\leq d \}$}
    \State {$\col(i)=\min\{ k>0:k \neq \col(j) \text{ for all } j\in W_i \}$}
    \EndFor
    \end{algorithmic}
\end{algorithm}

One can obtain different colorings depending on the order of the nodes; sorting the nodes by descending degree usually leads to a good performance~\cite{SchimmelThesis,Kubale04book}. 
The number of colors obtained with this algorithm is at most $\maxdeg^d+1$, where $\maxdeg$ is the maximum degree of a node in the graph induced by $A$, and the cost is at most $\mathcal{O}(n \maxdeg^d)$ if the $d$-neighbors $W_i$ of each node $i$ are computed with a traversal of the graph. Alternatively, the greedy coloring can be obtained by computing $A^d$, which can be done using at most $2 \lfloor \log_2 d \rfloor$ matrix-matrix multiplications~\cite[Section 4.1]{Higham}. In our Matlab implementation of the probing method, we construct the greedy distance-$d$ coloring by computing $A^d$, since we found it to be faster than the graph-based approach; this is likely due to the more efficient Matlab implementation of matrix-matrix operations.

If $A$ is $\beta$-banded, i.e.~if $[A]_{ij}=0$ for $\abs{i-j}>\beta$, a simple distance-$d$ coloring is given by
\begin{equation}
	\label{eqn:coloring_banded}
	\col(i)=(i-1) \text{ mod } (d\beta+1)+1,\quad i=1,\dots,n.
\end{equation}
This coloring is optimal with $s=d\beta+1$ if all the entries within the band are nonzero. 
Note that to diagonalize a symmetric banded matrix one must first reduce it to a tridiagonal form, and this costs $O(\beta n^2)$ \cite{Schwarz68}. On the other hand, if we assume that quadratic forms with $f(A)$ can be approximated with a fixed number of Krylov iterations, the computation of each quadratic form costs $O(\beta n)$, and the cost for $\traceP_d(f(A))$ is approximatively $O(d\beta^2 n)$ and this can be much less than the cost for the diagonalization provided that the requested accuracy is not too high.
For a general sparse matrix $A$, as an alternative to Algorithm~\ref{algorithm:greedy_coloring} one can use the reverse Cuthill-McKee algorithm~\cite{CuthillMcKee} to reorder the nodes and reduce the bandwidth, and then use \eqref{eqn:coloring_banded} to find a distance-$d$ coloring~\cite{Frommer2021}. Depending on the structure of the graph induced by $A$ this can yield good results, but in many cases better colorings are obtained by using Algorithm \ref{algorithm:greedy_coloring}; see Example~\ref{example:comparison_bound_coloring}.

Regarding the accuracy of the approximation of $\trace(f(A))$ by $\traceP_d(f(A))$  we have the following result \cite[Theorem 4.4]{Frommer2021}.

\begin{theorem}[\cite{Frommer2021}]\label{thm:trace_probing_error_polyapprox}

	Let $A\in\R^{n\times n}$ be symmetric with spectrum $\sigma(A)\subset[a,b]$. Let $f(x)$ be defined over $[a,b]$ and let $\traceP_d(f(A))$ be the approximation \eqref{eqn:trace_probing} of $\trace(f(A))$ induced by a distance-$d$ coloring.  Then
	\begin{equation}
		\label{eqn:trace_probing_error_polyapprox}
		\lvert \trace(f(A))-\traceP_d(f(A)) \rvert \leq 2n\,E_d(f,[a,b]).
	\end{equation}

\end{theorem}

If $f(x)=-x\log(x)$, so that $\trace(f(A))=S(A)$, we have the following result.

\begin{corollary}

	\label{cor:entropy-probing-bound}
	Let $A\in\R^{n\times n}$ be symmetric with $\sigma(A)\subset[a,b]$, $0\leq a<b$, and put $\gamma=a/b$. Then
	\begin{equation}
		\label{eqn:entropy_probing_bound}
		\lvert S(A)-\traceP_d(-A\log(A)) \rvert 
		\leq 
		n\,b\,(1-\sqrt{\gamma})\frac{1+\gamma+2d\sqrt{\gamma}}{2(d^2-1)}\left( \frac{1-\sqrt{\gamma}}{1+\sqrt{\gamma}} \right)^{d},
	\end{equation}
for all $d\geq 2$. In particular, if $a=0$, we have
\begin{equation}
	\label{eqn:entropy_probing_bound_singular}
	\lvert S(A)-\traceP_d(-A\log(A))\rvert \leq \frac{nb}{2(d^2-1)},
\end{equation}
for all $d\geq 2$.
\end{corollary}
\begin{proof}
	The inequality \eqref{eqn:entropy_probing_bound} follows from Theorem \ref{thm:entropy_polyapprox} and Theorem~\ref{thm:trace_probing_error_polyapprox}. The inequality \eqref{eqn:entropy_probing_bound_singular} follows from \eqref{eqn:entropy_probing_bound} with $\gamma=a/b=0$.
\end{proof}
\begin{remark}
	The bound \eqref{eqn:entropy_probing_bound} can be pessimistic in practice, especially for large values of $d$. We will see this in Example~\ref{example:comparison_bound_coloring} and in Section \ref{subsec:experiments--probing-bound-vs-estimate}. A priori bounds based on polynomial approximations often fail to catch the exact convergence behavior in many problems related to matrix functions since a minimization problem over the spectrum of a matrix is relaxed to the whole spectral interval. This occurs in the convergence of polynomial Krylov methods \cite[Section 5.6]{LiesenStrakos} and in the decay bounds on the entries of matrix functions \cite{BenziRinelli,FrommerInverse}. Also, the coloring we get via Algorithm \ref{algorithm:greedy_coloring} can return far more colors than needed for a distance-$d$ coloring, and this can benefit the convergence in a way that is not predicted by the bound. We will see in Section \ref{subsec:probing-method-implementation} a more practical heuristic to predict the error with higher accuracy.
\end{remark}

\begin{example}
	\label{example:comparison_bound_coloring}
	
	Let us see how the bound and the convergence of the probing method perform in practice. 
	We use the density matrix $\rho=\lap/\trace(\lap)$, where $\lap$ is the Laplacian of the graph \texttt{minnesota} from the SuiteSparse Matrix Collection \cite{DavisSparseCollection}. More precisely, we consider the biggest connected component whose graph Laplacian has size $n=2640$ and $9244$ nonzero entries. 
	
	For several values of $d$, we compute the approximation $\traceP_d(-\rho\log \rho)$ associated with two different distance-$d$ colorings: the first is obtained by the greedy coloring (Algorithm~\ref{algorithm:greedy_coloring}) after sorting the nodes by descending degree,  while for the second we use the reverse Cuthill-McKee algorithm to get a $67$-banded matrix and then \eqref{eqn:coloring_banded}. 
	
	In Figure \ref{fig:traceest-example-probing} we compare $\traceP_d(-\rho\log(\rho))$ with the value of $S(\rho)$ obtained by diagonalizing $\rho$, considered as exact. The left plot shows the error in terms of the number of colors (i.e.~the number of probing vectors) associated with the distance-$d$ colorings. In the right plot the errors are shown in terms of $d$ together with bound \eqref{eqn:entropy_probing_bound_singular}, where $b=\lambda_{\max}(\rho)$.

	For a fixed value of $d$, the coloring based on the bandwidth provides a smaller error than the greedy one, but it also uses a larger number of colors. 
	Indeed, we can see from the left plot that with the same computational effort the greedy algorithm obtains a smaller error.
	On the right we can observe that bound \eqref{eqn:entropy_probing_bound_singular} is close to the error given by the greedy coloring for small values of $d$, while it fails to catch the convergence behavior for large values of $d$. 

\end{example}

\begin{figure}[htbp]
	\makebox[\linewidth][c]{
	\begin{subfigure}[t]{.50\textwidth}
		\includegraphics[width=\textwidth]{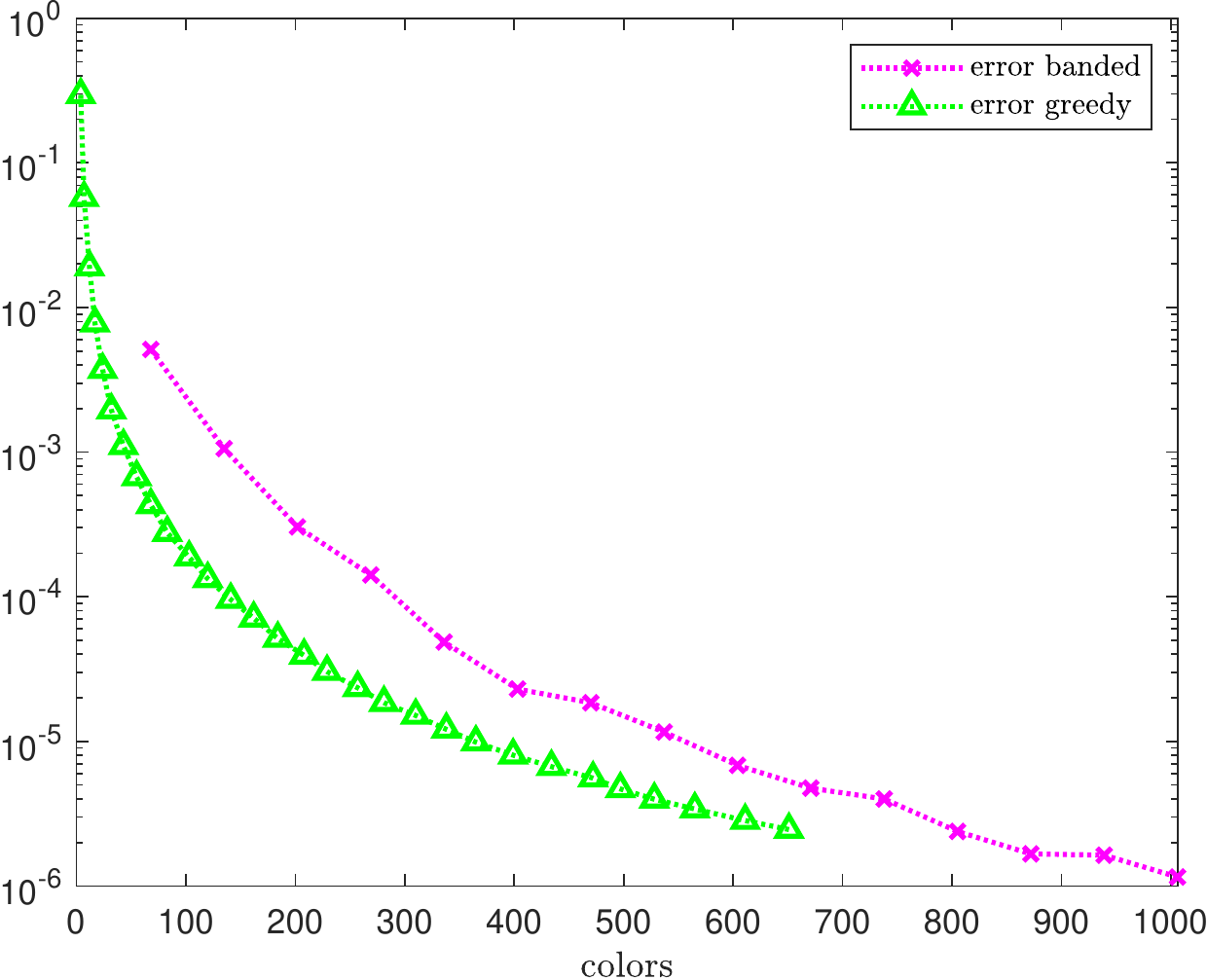}
	\end{subfigure}
	\begin{subfigure}[t]{.50\textwidth}
		\includegraphics[width=\textwidth]{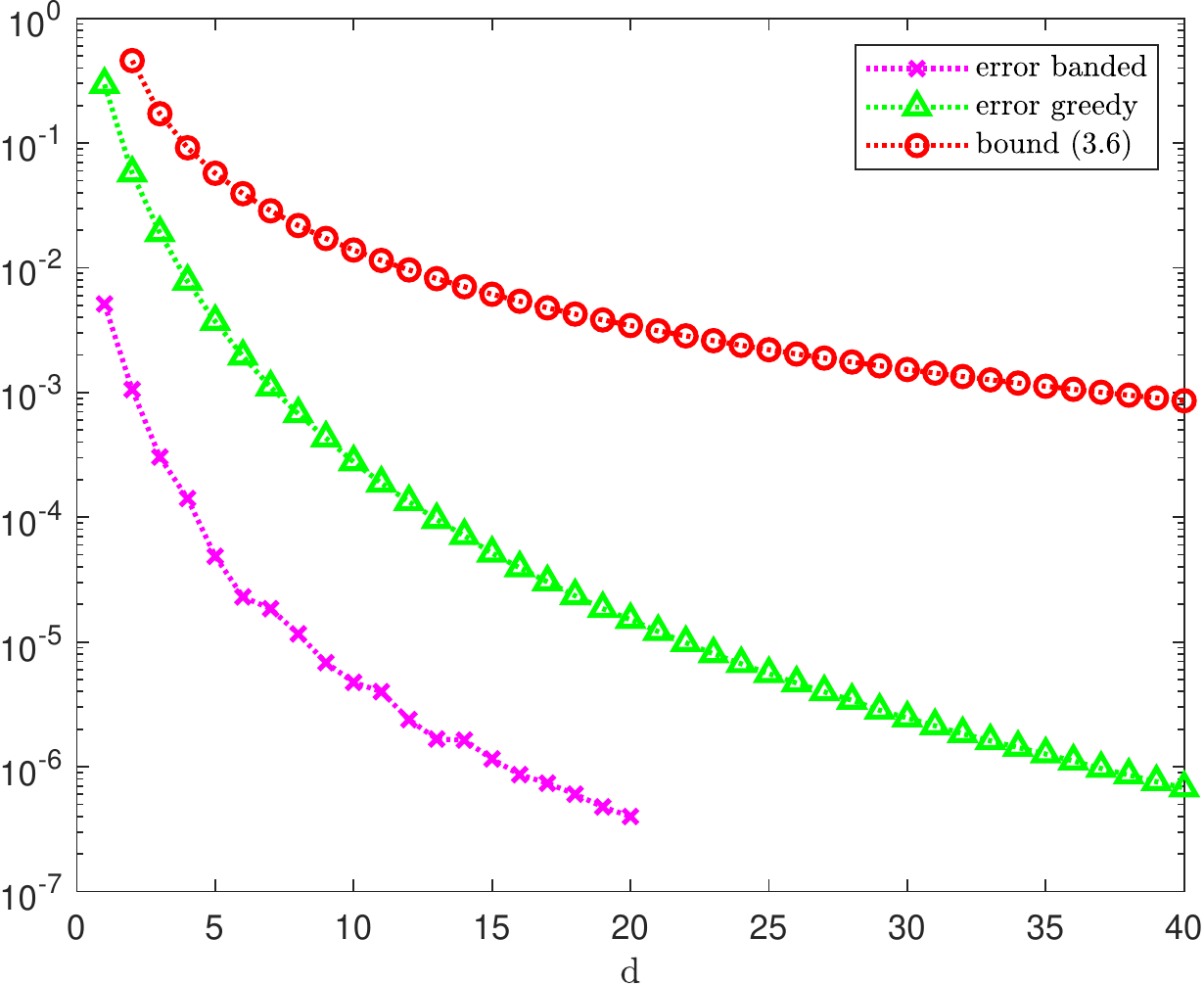}
	\end{subfigure}
	}
	\caption{
	Absolute errors of the probing approximation of $S(\rho)$ where the distance-$d$ coloring is obtained either by the greedy procedure (Algorithm~\ref{algorithm:greedy_coloring}) or with the reverse Cuthill-McKee algorithm and the coloring \eqref{eqn:coloring_banded} for banded matrices. On the left the abscissa represents the number of colors used for the coloring. On the right the errors are compared with bound \eqref{eqn:entropy_probing_bound_singular} in terms of $d$. 	
	\label{fig:traceest-example-probing}}
\end{figure}

For the graph entropy, i.e.~when $\rho=\lap/\trace(\lap)$  and $\lap$ is a graph Laplacian of a graph $\graph$, we can show that $-\traceP_d(\rho\log \rho)$ is a lower bound for $S(\rho)$. In the proof we use the fact that $\rho$ is a symmetric $M$-matrix~\cite[Chapter~6]{BermanMMatrix}, i.e.~it can be written in the form $\rho=\theta I-B$ where $\theta >0$ and $B$ is a nonnegative matrix such that $\lambda \leq \theta$ for all $\lambda\in\sigma(B)$.

\begin{lemma} 
	\label{lemma:entropy_M-matrix}
	Let $A=\theta I-B$ be a symmetric $M$-matrix. Let $\dist(i,j)$ be the geodesic distance in the graph associated with $A$. Then
		$[-A\log(A)]_{ij}\leq 0$ whenever~$\dist(i,j)\geq 2$. If $\theta<\exp(-1)$, then $[-A\log(A)]_{ij}\leq 0$ for $i\neq j$ and $-A\log A$ is an $M$-matrix.
\end{lemma}

\begin{proof}
	Suppose that $\theta>\|B\|_2=\lambda_{\max}(B)$, 
	so $A$ is nonsingular. 
	Then $f(x)=-x\log x$ 
	is analytic for 
	$\abs{x-\theta}<\|B\|_2$
	and it can be represented by the power series 
	\begin{equation*}
		f(x)=\sum_{k=0}^\infty \frac{f^{(k)}(\theta)}{k!}(x-\theta)^k.
	\end{equation*}
	Since $\sigma(A)\subset(\theta-\|B\|_2,\theta+\|B\|_2)$, the series expansion holds for $f(A)$:
	\begin{equation*}
		f(A)=
		\sum_{k=0}^\infty \frac{f^{(k)}(\theta)}{k!}(A-\theta I)^k =
		-\theta\log \theta I +(1+\log \theta)B - \sum_{k=2}^\infty \frac{\theta^{-(k-1)}}{k(k-1)}B^k.
	\end{equation*} 
	If $\dist(i,j)\geq 2$ we have $[B]_{ij}=[-A]_{ij}=0$, and $B^k$ is nonnegative for all $k \ge 0$, so $[f(A)]_{ij}\leq 0$. If $\theta < \exp(-1)$, then $(1+\log \theta)B$ is nonpositive so $[f(A)]_{ij}\leq 0$ for $i\neq j$, and it is an $M$-matrix since it is positive semidefinite \cite[Theorem 4.6]{BermanMMatrix}.

	If $A$ is singular, consider the nonsingular $M$-matrix $A+\epsilon I$ for $\epsilon>0$. 
	Then the results hold for $f(A+\epsilon I)$ (we impose $\theta+\epsilon<\exp(-1)$ if 
	$\theta<\exp(-1)$) and notice that $f(A)=\lim_{\epsilon\to 0}f(A+\epsilon I)$. This concludes the proof, since the limit of nonsingular $M$-matrices is an $M$-matrix.
\end{proof}

\begin{proposition}
	\label{prop:probing-graph-entropy-lower-bound}
	Let $A\in\R^{n\times n}$ be a symmetric $M$-matrix, and let $\traceP_d(-A\log(A))$ be the approximation \eqref{eqn:trace_probing} of $S(A)$ induced by a distance-$d$ coloring of $\graph(A)$ with $d\geq 1$. Then $\traceP_d(-A\log(A))\leq S(A)$.
\end{proposition}

\begin{proof}

	If $V_1, \dots, V_s$ is the graph partitioning associated with a distance-$d$ coloring, the error of the approximation can be written as
	\begin{equation}
		\label{eqn:probing_error_identity}
		S(A)-\traceP_d(-A\log(A))=-\sum_{\ell=1}^s\sum_{\substack{i,j\in V_\ell\\i\neq j}}[-A\log(A)]_{ij};
	\end{equation}
	see \cite{Frommer2021} for more details. By definition of a distance-$d$ coloring, for all $i,j\in V_\ell$, $i\neq j$, we have $\dist(i,j)\geq d+1 \geq 2$. Then, in view of Lemma \ref{lemma:entropy_M-matrix}, the right-hand side of \eqref{eqn:probing_error_identity} is nonnegative.
\end{proof}

\begin{remark}
	In this work we only consider the entropy of sparse density matrices. However, an important case is given by $\rho = g(H)$, where $H$ is the Hamiltonian of a certain quantum system and $g$ is a function defined on the spectrum of $H$. Notable examples are the Gibbs state \cite{Cramer,Wehrl} and the Fermi-Dirac state \cite{Aarons}. Despite $\rho$ being a dense matrix in general, a sparse structure of $H$ implies that $\rho$ exhibits decay properties and can be well approximated by a sparse matrix \cite{BenziCime,BenziBoitoRazouk}. Moreover, since $S(\rho)=-g(H)\log(g(H))$ one can apply the techniques described in this section to the composition $-g(x)\log(g(x))$. Although the investigation of this problem is outside the scope of this work, it represents an interesting direction for future research. These considerations also hold for the randomized techniques discussed below.
\end{remark}

\subsection{Stochastic trace estimation}
\label{subsec:stochastic-trace-estimation}

Let us consider the problem of computing $\trace(B)$, where $B \in \R^{n \times n}$ is a matrix that is not explicitly available but can be accessed via matrix-vector products $B \vec x$ and quadratic forms~$\vec x^T B \vec x$, for $\vec x \in \R^n$. The case of $B = f(A)$ can be seen as an instance of this problem, since $f(A)$ is expensive to compute, but~$f(A) \vec x$ and $\vec x^T f(A) \vec x$ can be efficiently approximated using Krylov methods, as we recall in Section~\ref{sec:krylov-methods}.

Stochastic trace estimators compute approximations of $\trace(B)$ by making use of the fact that, for any matrix $B$ and any random vector $\vec x$ such that $\expec [\vec x \vec x^T] = I$, we have $\expec [\vec x^T B \vec x] = \trace(B)$, where $\expec$ denotes the expected value. Hutchinson's trace estimator~\cite{Hutchinson89} is a simple stochastic estimator that generates $N$ vectors $\vec x_1, \dots, \vec x_N$ with i.i.d.~random $\mathcal{N}(0, 1)$ entries and approximates~$\trace(B)$ with
\begin{equation}
	\label{eqn:hutchinson-stochastic-trace-estimator}
	\trace_N^\text{Hutch}(B) = \frac{1}{N} \sum_{j = 1}^N \vec x_j^T B \vec x_j = \frac{1}{N} \trace(\vec X^T B \vec X), \qquad \vec X = [\vec x_1, \dots, \vec x_N].
\end{equation}
An algorithm that improves the convergence properties of Hutchinson's estimator has been proposed in \cite{MMMW21-Hutch++} with the name of \emph{Hutch++}. It samples $\Omega\in\R^{n\times N_r}$ with random i.i.d. $\mathcal{N}(0,1)$ entries, then computes $B\Omega$ and an orthonormal basis $Q\in\R^{n\times N_r}$ of $\range(B\Omega)$. Then the estimator is given by 
\begin{equation}\label{eqn:Hutch++_formula}
	\trace^{\text{Hutch++}}_{N_r,N_H}(B):=\trace(Q^TBQ)+\trace_{N_H}^{\text{Hutch}}((I-QQ^T)B(I-QQ^T)).
\end{equation}
This estimator computes the trace of a rank $N_r$ approximation of $B$ and estimates the trace of the remaining part using the Hutchinson estimator with $N_H$ samples, so its total cost is $N_r$ matrix-vector products and $N_r+N_H$ quadratic forms with $B$. It has been proven in \cite[Theorem~4.1]{MMMW21-Hutch++} that the complexity of Hutch++ is optimal up to logarithmic factors amongst algorithms that access a positive semidefinite matrix~$B$ via matrix-vector products.

The implementation of Hutch++ proposed in \cite{CortinovisKressner22} allows to prescribe a target accuracy $\epsilon>0$ and a failure probability $\delta\in(0,1)$ in input and then choose adaptively the parameters $N_r$ and $N_H$ in order to get the tail bound
\begin{equation} \label{eqn:tail-bound_Hutch++}
	\prob \big[\abs{\trace_{N_r,N_H}^{\text{Hutch++}}(B) - \trace(B)} \geq \epsilon\big] \leq \delta
\end{equation}
with as little computational effort as possible. This implementation goes under the name of \emph{adaptive Hutch++} and will be our stochastic estimator of choice for the von Neumann entropy in view of its flexibility and optimal convergence properties. Note that most of our theory for Krylov methods in Section \ref{sec:krylov-methods} can be used in combination with any other stochastic estimator based on matrix-vector products and quadratic forms with $B=f(A)$. 
For instance, we mention~\cite{SaibabaIpsen} where a low rank approximation of $A$ constructed using powers of $A$ is used to get a good approximation in case of fast decaying eigenvalues or large spectral gaps, and the paper~\cite{ChenHallman22} in which a Krylov subspace projection is integrated with the low rank approximation.
Furthermore, our theory in Section \ref{sec:krylov-methods} also applies to the techniques presented in the recent preprint \cite{XTrace} which improves on standard Hutch++ but does not allow an adaptive choice of the parameters as in adaptive Hutch++.

\section{Computation of quadratic forms with Krylov methods}
\label{sec:krylov-methods}

As shown in Section~\ref{sec:trace-estimation}, the approximate computation of $\trace(f(A))$ with probing methods or stochastic trace estimators can be reduced to the computation of several quadratic forms with $f(A)$, i.e.~expressions of the form $\vec b^T f(A) \vec b$. In this section, we briefly describe how they can be efficiently computed using polynomial and rational Krylov methods.

A \textit{polynomial Krylov subspace} associated to $A$ and $\vec b$ is given by
\begin{equation*}
	\kryl_m(A,\vec b) = \vspan \left \{ \vec b, A \vec b, \dots, A^{m-1} \vec b \right \} = \{ p(A) \vec b : p \in \poly_{m-1}\}.
\end{equation*}
More generally, given a sequence of poles $\{ \xi_j \}_{j \ge 1} \subset (\C\cup \{\infty\}) \setminus {\sigma(A)} \cup \{0\} $, we can define a \textit{rational Krylov subspace} as follows,
\begin{equation}
	\label{eqn:rational-krylov-subspace-definition}
	\rat_m(A,\vec b) = q_{m-1}(A)^{-1} \kryl_m(A, \vec b) = \Big\{ r(A) \vec b : r(z) = \frac{p_{m-1}(z)}{q_{m-1}(z)}, \text{with } p_{m-1} \in \poly_{m-1}\Big\},
\end{equation}
where $q_{m-1}(z) = \dprod_{j = 1}^{m-1}(1 - z/\xi_j)$. If all poles are equal to $\infty$, we have $q_{m-1}(z) \equiv 1$ and $\rat_m(A, \vec b)$ coincides with the polynomial Krylov subspace $\kryl_m(A, \vec b)$, so $\kryl_m(A, \vec b)$ can be considered as a special case of $\rat_m(A, \vec b)$. Note that this definition of $q_{m-1}$ does not allow us to have poles $\xi_j = 0$; this can be fixed by changing the definition of $q_{m-1}$ but it is not required in our case, since we are only going to use real negative poles and poles at $\infty$.

Let us denote by $V_m = [\vec v_1 \, \dots \, \vec v_m ]$ a matrix with orthonormal columns that spans the Krylov subspace $\rat_m(A, \vec b)$, and by $A_m = V_m^T A V_m$ the projection of $A$ onto the subspace. We can then project the problem on $\rat_m(A, \vec b)$ and approximate $\quadform = \vec b^T f(A) \vec b$ in the following way,
\begin{equation*}
	\quadform \approx \quadform_m = \vec b^T V_m f(A_m) V_m^T \vec b.
\end{equation*} 
If the basis $V_m$ is constructed incrementally using the rational Arnoldi algorithm~\cite{Ruhe94}, we have $\vec v_1 = \vec b / \norm{\vec b}_2$ and therefore 
\begin{equation}
\label{eqn:rational-krylov-approximation-quadform}
	\quadform_m =  \norm{\vec b}_2^2 \vec e_1^T f(A_m) \vec e_1.
\end{equation}
Note that the approximation $\quadform_m$ is closely related to the rational Krylov approximation of $f(A) \vec b$, which is given by
\begin{equation}
	\label{eqn:rational-krylov-approximation-matvec}
	f(A) \vec b \approx V_m f(A_m) V_m^T \vec b = \norm{\vec b}_2 V_m f(A_m) \vec e_1,
\end{equation}
and is known to converge with a rate determined by the quality of rational approximations of~$f$~\cite[Corollary~3.4]{Guettel13}.
We refer to~\cite{GuettelThesis, Guettel13} for an extensive discussion on rational Krylov methods for the computation of matrix functions. The approximation~\eqref{eqn:rational-krylov-approximation-quadform} can be also interpreted in terms of rational Gauss quadrature rules, see for instance~\cite{APR22, PranicReichel14}.

\begin{remark}
	The standard Arnoldi algorithm is inherently sequential since the computation of the new vector of the Krylov basis $\vec v_{m+1}$ requires the previous computation of $\vec v_m$. It is possible to parallelize it by solving several linear systems simultaneously and expanding the Krylov basis with blocks of vectors, with one of the strategies presented in \cite{BerljafaGuettel17}, at the cost of lower numerical stability. Since in this work we are expected to compute several quadratic forms $\vec b^T f(A) \vec b$, we can easily achieve parallelization by assigning the quadratic forms to different processors and thus we can neglect parallelism inside the computation of a single quadratic form.
\end{remark}

\begin{remark}
	\label{rem:rational-lanczos}
	We mention that for symmetric $A$ it is possible to construct the Krylov basis~$V_m$ using a method based on short recurrences such as rational Lanczos \cite{PalittaPozzaSimoncini22}. This has the advantage of reducing the orthogonalization costs, which can become significant if $m$ is large, and also avoids the need to store the matrix $V_m$ when approximating the quadratic form $\vec b^T f(A) \vec b$, see~\eqref{eqn:rational-krylov-approximation-quadform}. 
	However, the implementation in finite arithmetic of short recurrence methods can suffer from loss of orthogonality, which in turn can lead to a slower convergence. In order to avoid this potential problem, we use the rational Arnoldi method with full orthogonalization. Since we expect to attain convergence in a small number of iterations, the orthogonalization costs remain modest compared to the cost of operations with $A$.
\end{remark}

\subsection{Convergence}
By~\cite[Corollary~3.4]{Guettel13}, the accuracy of the approximation~\eqref{eqn:rational-krylov-approximation-matvec} for $f(A) \vec b$ is related to the quality of rational approximations to the function $f$ of the form $r(z) = q_{m-1}(z)^{-1} p_{m-1}(z)$, where $p_{m-1} \in \poly_{m-1}$ and $q_{m-1}$ is determined by the poles of the rational Krylov subspace~\eqref{eqn:rational-krylov-subspace-definition}.

In the case of quadratic forms we can prove a faster convergence rate using the fact that $\quadform_m = \quadform$ for rational functions of degree up to $(2m-1, 2m-2)$. 

\begin{lemma}	
	\label{lemma:exactness-rational-krylov-quadform}
	Assume that $A$ is symmetric. Let $p_{2m-1} \in \poly_{2m-1}$ and define the rational function $r(z) = q_{m-1}(z)^{-2} p_{2m-1}(z)$. Then we have
	\begin{equation*}
		\vec b^T r(A) \vec b = \vec b^T V_m r(A_m) V_m^T \vec b.
	\end{equation*} 
\end{lemma}
\begin{proof}
	It is sufficient to prove this fact for $p_{2m-1}(z) = z^k$, for $k = 0, \dots, 2m-1$. Assuming for the moment that $k = 2j + 1$ is odd, we have
	\begin{equation*}
		\vec b^T r(A) \vec b = \vec b^T s(A) A s(A) \vec b, \qquad \text{with} \quad s(z) = q_{m-1}(z)^{-1} z^{j}, \quad j < m.
	\end{equation*}
	Now using \cite[Lemma~3.1]{Guettel13}, we obtain
	\begin{equation*}
		\vec b^T r(A) \vec b = \big(\vec b^T V_m s(A_m) V_m^T \big) A \big(V_m s(A_m) V_m^T \vec b \big) = \vec b^T V_m^T r(A_m) V_m^T \vec b.
	\end{equation*}
	The case of $p_{m-1}(z) = z^k$ with $k$ even can be proved in the same way, by writing
	\begin{equation*}
		\vec b^T r(A) \vec b = \vec b^T s(A)^2 \vec b, \quad \text{with} \quad s(z) = q_{m-1}(z)^{-1} z^j, \quad j < m.
	\end{equation*}
\end{proof}

Lemma~\ref{lemma:exactness-rational-krylov-quadform} leads to the following convergence result for the approximation of quadratic forms, with the same proof as~\cite[Corollary~3.4]{Guettel13}.
\begin{proposition}
	\label{prop:convergence-rational-krylov-quadform}
	Let $A$ be symmetric with spectrum contained in $[\lambda_\text{min}, \lambda_\text{max}]$, $\quadform = \vec b^T f(A) \vec b$ and denote by $\quadform_m$ the approximation~\eqref{eqn:rational-krylov-approximation-quadform}. We have
	\begin{equation*}
		\abs{\quadform - \quadform_m} \le 2 \norm{\vec b}_2^2 \min_{p \in \poly_{2m-1}}\norm{f - q_{m-1}^{-2}p}_{[\lambda_\textnormal{min}, \lambda_\textnormal{max}]},
	\end{equation*}
\end{proposition}

By comparing Proposition~\ref{prop:convergence-rational-krylov-quadform} with~\cite[Corollary~3.4]{Guettel13}, we can expect the convergence for quadratic forms to be roughly twice as fast as the one for matrix-vector products with $f(A)$. 

\subsection{Poles for the rational Krylov subspace}
\label{subsec:rational-krylov-poles}

Recall that the function $f(z) = x \log x$ has the integral expression~\eqref{eqn:entropy-integral-expression-cs}, which corresponds to a Cauchy-Stieltjes function multiplied by the polynomial $x(1-x)$. This implies that we can expect that a pole sequence that yields fast convergence for Cauchy-Stieltjes functions will be also effective in our case, especially if we add two poles at $\infty$ to account for the degree-two polynomial.

The authors of~\cite{MasseiRobol21} consider the case of a positive definite matrix $A$ with spectrum in $[a, b]$ and a Cauchy-Stieltjes function $f$, and relate the error for the computation of $f(A) \vec b$ with a rational Krylov method to the \textit{third Zolotarev problem} in approximation theory. 
The solution to this problem is known explicitly and it can be used to find poles on $(-\infty, 0)$ that provide in some sense an optimal convergence rate for the rational Krylov method~\cite[Corollary~4]{MasseiRobol21}.
However, the optimal Zolotarev poles are not nested, so they cannot be used to expand the Krylov subspace incrementally, and they are practical only if one knows in advance how many iterations to perform, for instance by relying on an a priori error bound. 
This drawback can be overcome by constructing a nested sequence of poles that is equidistributed according to the limit measure identified by the optimal poles, which can be done with the method of equidistributed sequences (EDS) described in~\cite[Section~3.5]{MasseiRobol21}.  
These poles have the same asymptotic convergence rate as the optimal Zolotarev poles and are usually better for practical purposes. To be computed, they require the knowledge of $[a, b]$ or a positive interval $\Sigma$ such that~$[a, b] \subset \Sigma$.

As an alternative, one can also use poles obtained from Leja-Bagby points~\cite{Bagby69,Guettel13}. These points can be computed with an easily implemented greedy algorithm and they have an asymptotic convergence rate that is close to the optimal one. See~\cite[Section~4]{Guettel13} and the references therein for additional information.

\begin{remark}
	\label{rem:mixing-polynomial-and-rational-iterations}
	For the function $f(x) = x \log x$, the first few iterations of a polynomial Krylov method have a fast convergence rate that is close to the convergence rate of rational Krylov methods, even if it becomes asymptotically much slower for ill conditioned matrices. The faster initial convergence can be explained by the algebraic factor in the bound~\eqref{eqn:entropy_polyapprox_integral_explicit} for polynomial approximations of $x \log x$. Since polynomial Krylov iterations are cheaper than rational Krylov iterations, this suggests the use of a mixed polynomial-rational method, that starts with a few polynomial Krylov steps and then switches to a rational Krylov method with, e.g., EDS poles to achieve a higher accuracy. These methods are compared numerically in Example~\ref{example:krylov-bound}, where we also test the performance of the a posteriori error bound that we prove in Section~\ref{sec:a-posteriori-error-bound}.
\end{remark}

\begin{remark}
	\label{rem:graph-laplacian-krylov}
	Note that in the context of the graph entropy the matrix $A$ is a graph Laplacian, which is a singular matrix. 
	Therefore in principle it is not possible to use the poles described in this section, since here we assume that $A$ is positive definite. 
	However, we can use one of the desingularization strategies described in~\cite{BenziSimunec} to remove the $0$ eigenvalue of the graph Laplacian, obtaining a matrix with spectrum contained in $[\lambda_2, \lambda_n]$, where $\lambda_2$ is the second smallest eigenvalue of $A$ and $\lambda_n$ is the largest one. 
	In our implementation we use the approach that is called implicit desingularization in~\cite{BenziSimunec}, which consists in replacing the initial vector $\vec b$ for the Krylov subspace with $\vec c = \vec b - \frac{\one^T \vec b}{n} \one$, where $\one$ is the vector of all ones. 
	Since $\vec c$ is orthogonal to the eigenvector $\one$ associated to the eigenvalue $0$, it can be shown that the convergence of a Krylov subspace method with starting vector $\vec c$ is the same as for a matrix with spectrum in $[\lambda_2, \lambda_n]$. 
	An approximation of $f(A) \vec b$ can be then cheaply recovered from $f(A) \vec c$ using the fact that $f(A) \one = f(0) \one$, and similarly for $\vec b^T f(A) \vec b$. See~\cite{BenziSimunec} for more details.
\end{remark}

\subsection{A posteriori error bound}
\label{sec:a-posteriori-error-bound}

In this section we prove an a posteriori bound for the error in the computation of the quadratic form $\vec b^T f(A) \vec b$ with a rational Krylov method. This bound is a variant of the one described in \cite[Section~6.6.2]{GuettelThesis} for $f(A) \vec b$, modified in order to account for the faster convergence rate in the case of quadratic forms.

We recall that after $m$ iterations the rational Arnoldi algorithm yields the rational Arnoldi decomposition~\cite[Definition~2.3]{BerljafaGuettel15}
\begin{equation}
	\label{eqn:rational-arnoldi-decomposition}
	A V_{m+1} \underline{K_m} = V_{m+1} \underline{H_m},
\end{equation}
where $\vspan V_{m+1} = \rat_{m+1}(A, \vec b)$ and $\underline{K_m}$, $\underline{H_m}$ are $(m+1) \times m$ upper Hessenberg matrices with full rank. Let us consider the situation when $\xi_m = \infty$: in this case the last row of $\underline{K_m}$ is zero, and the decomposition simplifies to 
\begin{equation*}
	A V_m K_m = V_m H_m + \vec v_{m+1} \vec h_{m+1}^T,
\end{equation*}
where $H_m$ and $K_m$ denote the $m \times m$ leading principal blocks of $\underline{H_m}$ and $\underline{K_m}$, respectively, and $\vec h_{m+1}^T = h_{m+1, m} \vec e_m^T$ denotes the last row of $\underline{H_m}$. 
Note that $K_m$ is nonsingular since $\underline{K_m}$ has full rank, so we can rewrite the decomposition as
\begin{equation}
	\label{eqn:rational-arnoldi-decomposition--simplified}
	A V_m = V_m A_m + \vec v_{m+1} \vec h_{m+1}^T K_m^{-1}, \qquad \text{where} \quad A_m = V_m^T A V_m = H_m K_m^{-1}.
\end{equation}

\begin{remark}
	\label{rem:restricting-to-simplified-hessenberg-decomposition}
	To derive the bound, we assume that $\xi_m = \infty$ because it simplifies the rational Arnoldi decomposition and hence the expression of the bound. Such an assumption is not restrictive, since the value of $\xi_m$ does not have any impact on $V_m$ and $A_m$, but only on $\vec v_{m+1}$ and the last column of $\underline{H_m}$ and $\underline{K_m}$. As we shall see later, we can use a technique described in \cite[Section~6.1]{GuettelThesis} to compute the bound for all $m$, without having to set the corresponding poles $\xi_m = \infty$. Note that if $\xi_m \ne \infty$, then~\eqref{eqn:rational-arnoldi-decomposition--simplified} does not hold, and in particular $V_m^T A V_m \ne H_m K_m^{-1}$.
\end{remark}

By using the Cauchy integral formula for $f$, we can obtain the following expression for the error~\cite[Section~6.2.2]{GuettelThesis}:
\begin{equation}
	\label{eqn:error-cauchy-integral--guettel}
	f(A) \vec b - V_m f(A_m) V_m^T \vec b = \frac{1}{2 \pi i}\int_\Gamma f(z) (z I - A)^{-1} \vec r_m(z) \de z,
\end{equation}
where $\Gamma$ is a contour contained in the region of analyticity of $f$ that encloses the spectrum of~$A$, and 
\begin{equation*}
	\vec r_m(z) = \vec b - (z I - A) \vec x_m(z), \qquad \text{with} \quad \vec x_m(z) = V_m (z I - A_m)^{-1} V_m^T \vec b,
\end{equation*}
which can be seen as a residual vector of the shifted linear system $(z I - A) \vec x = \vec b$. It turns out that~\cite[Section~6.2.2]{GuettelThesis}
\begin{equation*}
	\vec r_m(z) = \norm{\vec b}_2 \varphi_m(z) \vec v_{m+1}, \qquad \text{with} \quad \varphi_m (z) = \vec h_{m+1}^T K_m^{-1} (z I - A_m)^{-1} \vec e_1.
\end{equation*}
Observe that we have
\begin{align*}
	\vec b^T (z I - A)^{-1} \vec r_m(z) &= \vec r_m(z)^T (z I - A)^{-1} \vec r_m(z) + \vec x_m(z)^T \vec r_m(z) \\
	&= \vec r_m(z)^T (z I - A)^{-1} \vec r_m(z),
\end{align*}
where we exploited the fact that $\vec x_m(z) \in \rat_m(A, \vec b) \perp \vec r_m(z)$.

By using this property in conjunction with~\eqref{eqn:error-cauchy-integral--guettel}, we can write the error for the quadratic form $\vec b^T f(A) \vec b$ as 
\begin{equation}
	\label{eqn:error-cauchy-integral-quadform}
	\begin{aligned}
		\quadform - \quadform_m &= \frac{1}{2 \pi i} \int_\Gamma f(z) \vec r_m(z)^T (z I - A)^{-1} \vec r_m(z) \de z \\
		&=  \frac{1}{2 \pi i} \norm{\vec b}_2^2 \int_\Gamma f(z) \varphi_m(z)^2 \vec v_{m+1}^T (z I - A)^{-1} \vec v_{m+1} \de z.
	\end{aligned}
\end{equation}
We can now follow the same steps used in~\cite[Section~6.2.2]{GuettelThesis} to bound the integral in~\eqref{eqn:error-cauchy-integral-quadform}. Assume that $A_m$ has the spectral decomposition $A_m = U_m D_m U_m^T$, with $U_m$ orthogonal and $D_m = \diag(\theta_1, \dots, \theta_m)$, and define the vectors 
\begin{equation*}
	[\alpha_1, \dots, \alpha_m] = \vec h_{m+1}^T K_m^{-1} U_m \qquad \text{and} \qquad [\beta_1, \dots, \beta_m]^T = U_m^T \vec e_1,
\end{equation*}
so that we have
\begin{equation*}
	\varphi_m(z) = \vec h_{m+1}^T K_m^{-1} U_m (z I - D_m)^{-1} U_m^T \vec e_1 = \sum_{j = 1}^m \alpha_j \beta_j \frac{1}{z - \theta_j},
\end{equation*}
and
\begin{equation*}
	\varphi_m(z)^2 
	= \sum_{j = 1}^m \alpha_j^2 \beta_j^2 \frac{1}{(z - \theta_j)^2} + 2\sum_{j = 1}^m \alpha_j \beta_j \gamma_j \frac{1}{z - \theta_j},
\end{equation*}
where we defined $\gamma_j = \dsum_{\ell \,:\, \ell \ne j} \alpha_\ell \beta_\ell \dfrac{1}{\theta_j - \theta_\ell}$. 
By plugging this expression into \eqref{eqn:error-cauchy-integral-quadform} we get
\begin{align*}
	\frac{1}{\norm{\vec b}_2^2} (\quadform - \quadform_m) &= \frac{1}{2 \pi i} \int_\Gamma f(z) \varphi_m(z)^2 \vec v_{m+1}^T (z I - A)^{-1} \vec v_{m+1} \de z \\
	&= \sum_{j = 1}^m \alpha_j^2 \beta_j^2 \frac{1}{2 \pi i} \int_\Gamma \frac{f(z)}{(z - \theta_j)^2} \vec v_{m+1}^T (z I - A)^{-1} \vec v_{m+1} \de z \\
	& \quad + 2\sum_{j = 1}^m \alpha_j \beta_j \gamma_j \frac{1}{2 \pi i} \int_\Gamma \frac{f(z)}{z - \theta_j} \vec v_{m+1}^T (z I - A)^{-1} \vec v_{m+1}\de z \\
	&= \sum_{j = 1}^m \alpha_j^2 \beta_j^2 \vec v_{m+1}^T \Big( (f(A) - f(\theta_j) I)(A - \theta_j I)^{-2} - f'(\theta_j) (A - \theta_j I)^{-1} \Big) \vec v_{m+1} \\
	&\quad + 2\sum_{j = 1}^m \alpha_j \beta_j \gamma_j \vec v_{m+1}^T (f(A) - f(\theta_j) I)(A - \theta_j I)^{-1} \vec v_{m+1},
\end{align*}
where for the last equality we used the residue theorem~\cite[Theorem~4.7a]{Henrici88book}.

Let us define 
\begin{equation}
	\label{eqn:error-bound--gm-definition}
	g_m(z) = \sum_{j = 1}^m \begin{cases}
		\alpha_j^2 \beta_j^2 \left( \dfrac{f(z) - f(\theta_j)}{(z - \theta_j)^2} - \dfrac{f'(\theta_j)}{z - \theta_j} \right) + 2 \alpha_j \beta_j \gamma_j \dfrac{f(z) - f(\theta_j)}{z - \theta_j} & \text{if } z \ne \theta_j, \\
		\dfrac{1}{2} \alpha_j^2 \beta_j^2 f''(\theta_j) + 2 \alpha_j \beta_j \gamma_j  f'(\theta_j) & \text{if } z = \theta_j,
	\end{cases}
\end{equation}
where the expression for $z = \theta_j$ is obtained by taking the limit for $z \to \theta_j$ in the definition for~$z \ne \theta_j$.
The above computations immediately lead to the following a posteriori bound for the quadratic form error.

\begin{theorem}
	\label{thm:error-bound-a-posteriori}
	Let $A$ be a symmetric matrix with spectrum $\sigma(A) \subset [\lambda_\text{min}, \lambda_\text{max}]$. Using the same notation as above, we have
	\begin{equation}
		\label{eqn:error-bound-a-posteriori}
		\norm{\vec b}_2^2 \min_{z \in [\lambda_\text{min}, \lambda_\text{max}]} \abs{g_m(z)} \le \abs{\quadform - \quadform_m} \le \norm{\vec b}_2^2 \max_{z \in [\lambda_\text{min}, \lambda_\text{max}]} \abs{g_m(z)}.
	\end{equation}
\end{theorem}

\begin{remark}
	\label{rem:lower-and-upper-error-bound}
	We are mainly interested in the upper bound in~\eqref{eqn:error-bound-a-posteriori} to have a reliable stopping criterion for the rational Krylov method, although the lower bound can also be of interest. We also mention that other bounds and error estimates can be obtained, such as the other ones described in~\cite[Section~6.6]{GuettelThesis}, but we found that the one derived in this section worked well enough for our purposes.
	Under certain assumptions, it is also possible to obtain upper and lower bounds for the quadratic form $\vec b^T f(A) \vec b$ using pairs of rational Gauss quadrature rules, such as Gauss and Gauss-Radau quadrature rules. We refer to~\cite{APR22} for more details.
\end{remark}

\begin{example}
	\label{example:krylov-bound}
	In this example we test the accuracy of the lower and upper bounds given in~\eqref{eqn:error-bound-a-posteriori} for polynomial and rational Krylov methods. We consider the computation of $\vec b^T f(A) \vec b$, for a $2000 \times 2000$ matrix $A$ with eigenvalues given by the Chebyshev points for the interval $\Sigma = [10^{-3}, 10^3]$, a random vector $\vec b$ and $f(x) = x \log(x)$. The upper and lower bounds are computed numerically by evaluating $g_m$ on a discretization of the interval $[\lambda_\text{min}, \lambda_\text{max}]$. In addition to the lower and upper bounds, we also consider a heuristic estimate of the error given by the geometric mean of the upper and lower bound in~\eqref{eqn:error-bound-a-posteriori}, i.e.
	\begin{equation}
		\label{eqn:error-bound-a-posteriori--estimate}
		\texttt{est}_m = \norm{\vec b}_2^2 \displaystyle \sqrt{\min_{z \in \Sigma}  \abs{g_m(z)} \max_{z \in \Sigma}  \abs{g_m(z)}}.
	\end{equation}
	The results are shown in Figure~\ref{fig:krylov-example-2}. On the left plot we show the convergence for the polynomial Krylov method and for a rational Krylov method with poles from an EDS for Cauchy-Stieltjes functions, and on the right plot a mixed polynomial-rational method that uses $10$ poles at $\infty$ (which correspond to polynomial Krylov steps) followed by $10$ EDS poles (see Remark~\ref{rem:mixing-polynomial-and-rational-iterations}). 
	The upper and lower bounds match the shape of the convergence curve quite well, although they are less accurate when polynomial iterations are used. Rather surprisingly, the geometric mean of the bounds gives a very accurate estimate for the error, even in the case when the bounds themselves are less accurate.
\end{example}

\begin{remark}
	\label{rem:geometric-mean-error-estimate}
	We do not have a rigorous explanation for the accuracy of the estimate based on the geometric mean of the bounds in~\eqref{eqn:error-bound-a-posteriori}, but from further experiments it seems to be very accurate also for other functions. Unfortunately, if the spectral interval $\Sigma$ is known only approximately, the bounds become less tight and the geometric mean estimate usually ends up underestimating the actual error by one or two orders of magnitude.
\end{remark}

\begin{figure}[htbp]
	\makebox[\linewidth][c]{
	\begin{subfigure}[t]{.55\textwidth}
		\includegraphics[width=0.95\textwidth]{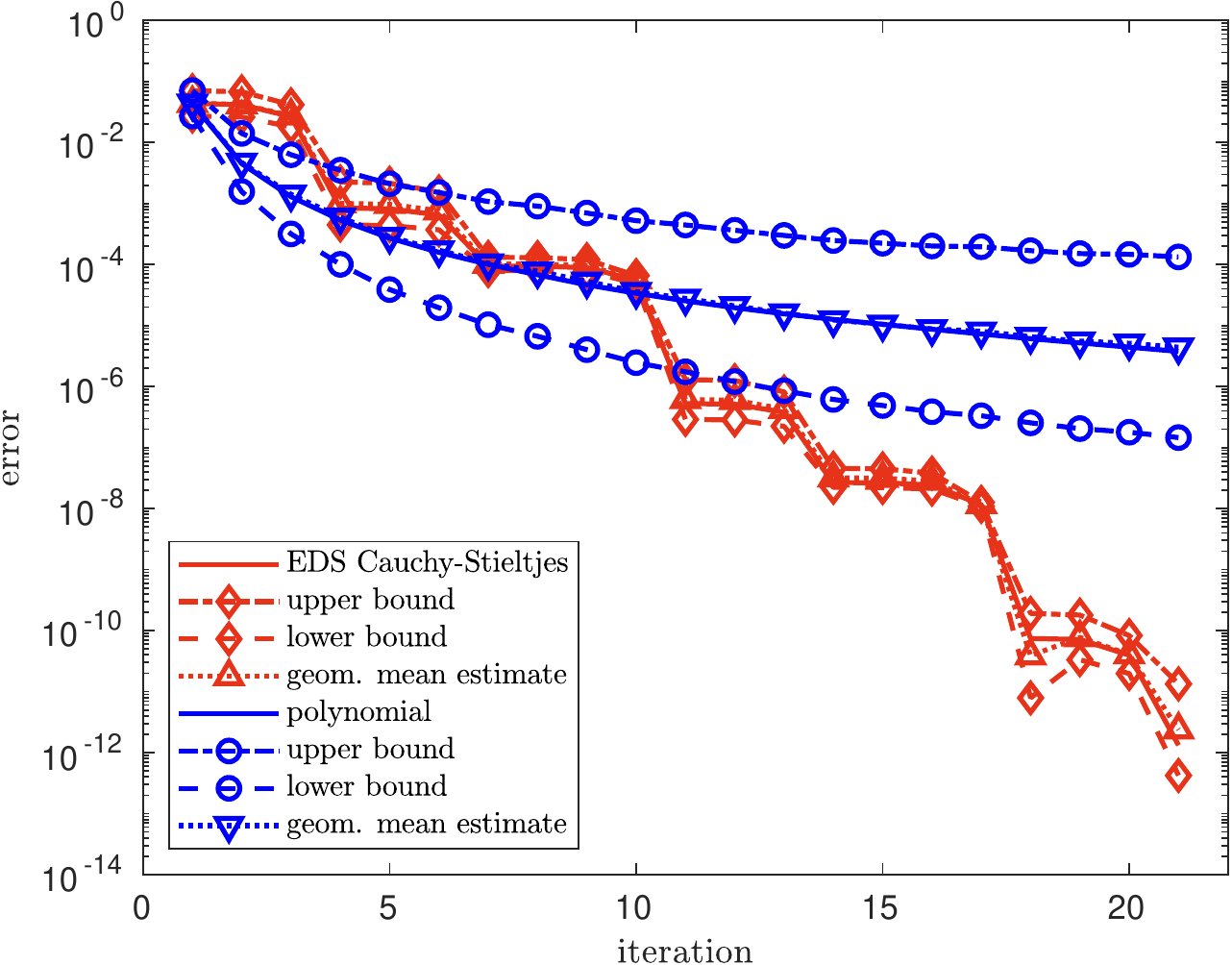}
	\end{subfigure}
	\begin{subfigure}[t]{.55\textwidth}
		\includegraphics[width=0.95\textwidth]{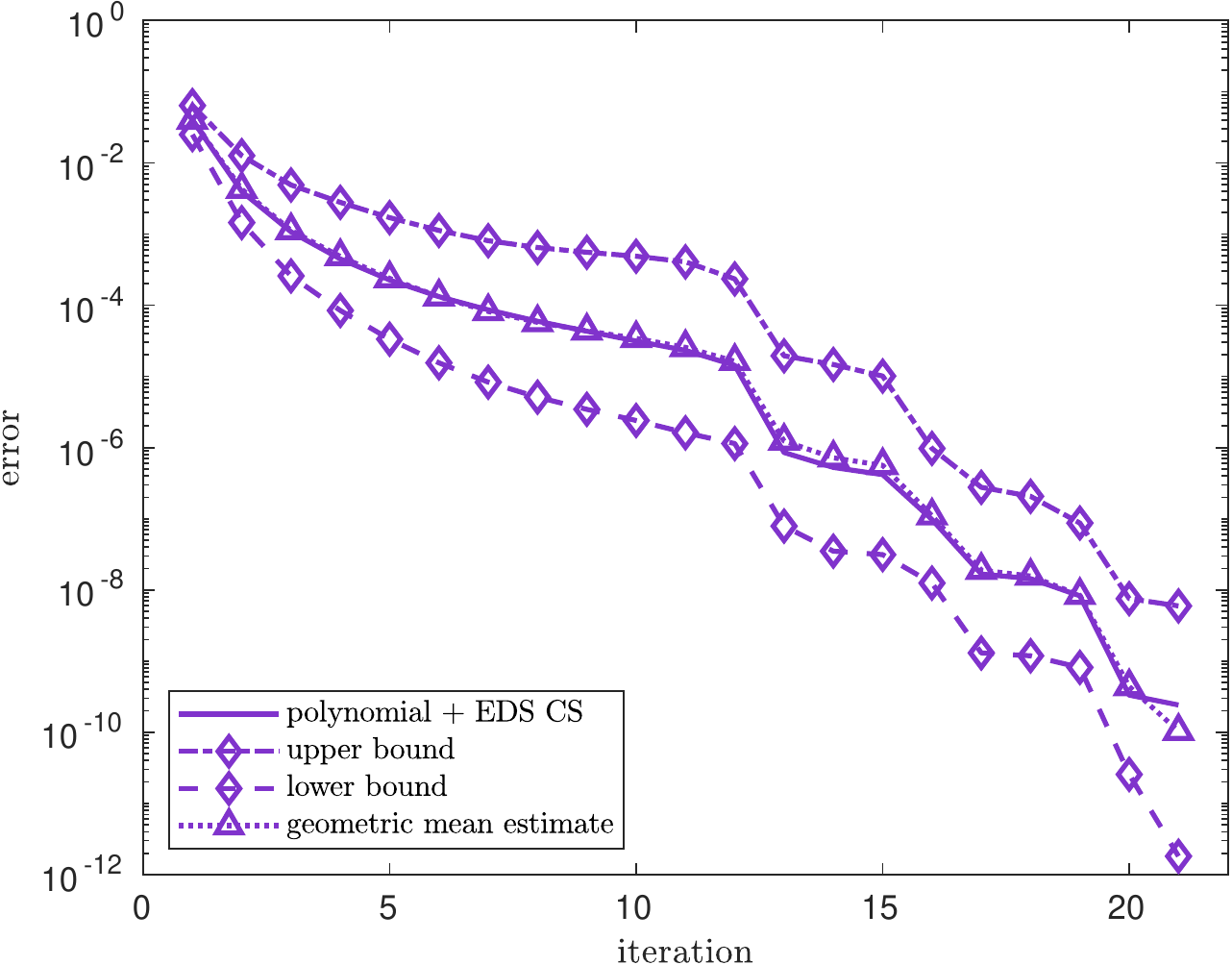}
	\end{subfigure}}
	\caption{Accuracy of error bounds and estimates for the relative error in the computation of~$\vec b^T f(A) \vec b$ with Krylov methods, where $\vec b$ is a random vector, $f(x) = x \log(x)$ and $A$ is a $2000 \times 2000$ matrix with eigenvalues that are Chebyshev points in the interval~$[10^{-3}, 10^3]$. Left: polynomial Krylov and rational Krylov with EDS poles for Cauchy-Stieltjes function. Right: $10$ poles at $\infty$ and 10 EDS poles for Cauchy-Stieltjes functions.}
	\label{fig:krylov-example-2}
\end{figure}

\subsubsection{Computation of the bound}
\label{subsubsec:krylov-bound-computation}

Recall that the a posteriori bounds in \eqref{eqn:error-bound-a-posteriori} hold after the $m$-th iteration only if $\xi_m = \infty$. In a practical scenario, i.e.~when using the bounds as a stopping criterion for a rational Krylov method, it is desirable to evaluate the bounds after each iteration, without being forced to set the corresponding pole to $\infty$.
As anticipated in Remark~\ref{rem:restricting-to-simplified-hessenberg-decomposition}, we provide here two approaches to evaluate the bounds in~\eqref{eqn:error-bound-a-posteriori} even when $\xi_m \ne \infty$.

One way to avoid setting poles to $\infty$, proposed in~\cite[Section~6.1]{GuettelThesis}, is to use an auxiliary basis vector $\vec v_\infty$, which is initialized as $\vec v_\infty^{(1)} = A \vec v_1$ at the beginning of the rational Arnoldi algorithm, and maintained orthonormal to the basis vectors $\{\vec v_1, \dots, \vec v_j\}$ at each iteration $j$, at the cost of only one additional orthogonalization per iteration. The basis $[V_j \; \vec v_\infty^{(j)}]$ is an orthonormal basis of the rational Krylov subspace $\rat_{j+1}(A, \vec b)$ with poles $\{ \xi_1, \dots, \xi_{j-1}, \infty \}$, and it is associated to the auxiliary Arnoldi decomposition
\begin{equation*}
	A V_j \widetilde{K}_j = [V_j \; \vec v_\infty^{(j)}] \underline{\widetilde{H}_j}, 
\end{equation*}
where $\widetilde{K}_j\vec e_j = \vec e_1$ and the last column of $\underline{\widetilde{H}_j}$ contains the orthogonalization coefficients for~$\vec v_\infty^{(j)}$. This decomposition can be used to compute the bound~\eqref{eqn:error-bound-a-posteriori} since the last row of $\underline{\widetilde{K}_j}$ is zero by construction.
\begin{remark}
	\label{rem:auxiliary-vector-trick-with-poles-at-infinity}
	We point out that if $\xi_j = \infty$ for some $j$, then the approach described above will not work from iteration $j+1$ onward, since $A \vec v_1 \in \rat_{j+1}(A, \vec b)$ and therefore $\vec v_\infty^{(j+1)} = \vec 0$ after orthogonalization. This is easily fixed by switching to a different auxiliary vector at iteration $j+1$, such as $\vec v_\infty = A \vec v_{j+1}$, or by setting directly at the start $\vec v_\infty = A^{\ell + 1} \vec v_1$, where $\ell$ is the number of poles at $\infty$ used to construct the rational Krylov subspace.
\end{remark}

In our experience, the technique described above can be sometimes subject to instability due to a large condition number of the matrix $\widetilde{K}_j$. We therefore also propose another approach, which is inspired by the methods for moving the poles of a rational Krylov subspace presented in~\cite[Section~4]{BerljafaGuettel15}. The idea is to add a pole at $\infty$ at the beginning of the pole sequence, and reorder the poles at each iteration in order to always have the last pole equal to $\infty$. 

First of all, we recall how to swap poles in a rational Arnoldi decomposition. This procedure is a special case of the algorithm described in \cite{BerljafaGuettel15}, but we still describe it in some detail for completeness. Recall that the poles of a rational Krylov subspace are the ratios of the entries below the main diagonals of $\underline{H_j}$ and $\underline{K_j}$~\cite[Definition~2.3]{BerljafaGuettel15}, i.e.~$\xi_j = h_{j+1, j}/k_{j+1, j}$. In other words, the poles $\{\xi_1, \dots, \xi_j\}$ are the eigenvalues of the upper triangular pencil $(\widehat{H}_j, \widehat{K}_j)$, where we denote by $\widehat{H}_j$ the bottom $j \times j$ block of $\underline{H_j}$, and similarly for $\widehat{K}_j$. So we can obtain a transformation that swaps two adjacent poles in the same way as orthogonal transformations that reorder eigenvalues in a generalized Schur form \cite{Kressner06}. 
Let $U_j$ and $W_j$ be $j \times j$ orthogonal matrices such that the pencil $U_j^T (\widehat{H}_j, \widehat{K}_j) W_j$ is still in upper triangular form and has the last two eigenvalues in reversed order; the matrices $U_j$ and $W_j$ only involve $2 \times 2$ rotations, and they can be computed and applied cheaply as described in \cite{Kressner06}.
Defining $\widehat{U}_j = \blkdiag(1, U_j) \in \R^{(j+1) \times (j+1)}$, we have the new rational Arnoldi decomposition
\begin{equation*}
	A \widetilde{V}_{j+1} \underline{\widetilde{K}_j} = \widetilde{V}_{j+1} \underline{\widetilde{H}_j},
\end{equation*}
where
\begin{equation*}
	\widetilde{V}_{j+1} = V_{j+1} \widehat{U}_j, \qquad \underline{\widetilde{K}_j} = \widehat{U}_j^T \underline{K_j} W_j \qquad \text{and} \qquad \underline{\widetilde{H}_j} = \widehat{U}_j^T \underline{H_j} W_j.
\end{equation*}
This decomposition has the same poles as $A V_{j+1} \underline{K_j} = V_{j+1} \underline{H_j}$, with the difference that the last two poles $\xi_{j-1}$ and $\xi_j$ are now swapped. In particular, if $\xi_{j-1} = \infty$, the last pole of the new decomposition is now $\infty$, and hence the last row of $\underline{\widetilde{K}_j}$ is equal to zero.

Given the pole sequence $\{ \xi_1, \xi_2, \dots \}$, let us consider the rational Krylov subspace associated to the modified pole sequence $\{ \infty, \xi_1, \xi_2, \dots\}$.
Clearly, after the first iteration both pole sequences identify the same subspace $\rat_1(A, \vec b)$, but the last (and first) pole of the modified sequence is~$\infty$, so the last row of $\underline{K_1}$ is zero and we can use the decomposition $A V_1 K_1 = V_2 \underline{H_1}$ to compute the bound~\eqref{eqn:error-bound-a-posteriori}.
After the second iteration, if $\xi_1 \ne \infty$, we can swap the poles $\xi_1$ and~$\infty$ with the procedure outlined above to obtain the decomposition $A V_2 K_2 = V_3 \underline{H_2}$ associated to the poles $\{ \xi_1, \infty \}$, where again the last row of $\underline{K_2}$ is equal to zero (for simplicity we still use the notation $K_j$ instead of $\widetilde{K}_j$, and similarly for $V_j$ and $H_j$). 
If $\xi_1 = \infty$, there is no need to swap poles and we can proceed to the next iteration. 

By repeating the same steps at each iteration, we can ensure that after $j$ iterations we have a decomposition $A V_j K_j = V_{j+1} \underline{H_j}$, associated to the poles $\{ \xi_1, \dots, \xi_{j-1}, \infty \}$ in this order, so that the last row of $\underline{K_j}$ is equal to zero and it can be used to compute the bound~\eqref{eqn:error-bound-a-posteriori}. 

\begin{remark}
	Note that $V_j$ is a basis of $\rat_j(A, \vec b)$, which is the same subspace that we would have obtained if we had run the rational Arnoldi algorithm with poles $\{\xi_1, \dots, \xi_{j-1}\}$; so the method described in this section actually computes the approximation $\quadform_j$ and the bound associated to the poles $\{\xi_1, \dots, \xi_{j-1}\}$, and not to the modified pole sequence $\{ \infty, \xi_1, \dots, \xi_{j-2} \}$. 
	The initial pole at~$\infty$ is only added to enable the computation of the bound and it is never used in the actual approximation.
\end{remark}

\section{Implementation aspects}
\label{sec:algorithm-overview}	

In this section we outline the algorithm used to compute the entropy obtained by connecting the different components presented in the previous sections, and we briefly comment on some of the decisions that have to be taken in an implementation, especially concerning stopping criteria. Given a symmetric positive semidefinite matrix $A$ with $\trace(A) = 1$ and a target relative accuracy~$\epsilon$, the algorithm should output an estimate \texttt{trest} of $\trace(f(A))$, where $f(x) = - x \log x$, such that
\begin{equation*}
	\abs*{\trace(f(A)) - \texttt{trest}} \le \epsilon \trace(f(A)),
\end{equation*}
using either the probing approach of Section~\ref{subsec:probing-methods} or a stochastic trace estimator from Section~\ref{subsec:stochastic-trace-estimation}. Observe that the entropy of an $n \times n$ density matrix is always bounded form above by $\log n$, but it may be in principle very small, so we prefer to aim for a certain relative accuracy rather than an absolute accuracy.
Quadratic forms and matrix-vector products with $f(A)$ are computed using Krylov methods, specifically using a certain number of poles at $\infty$ followed by the EDS poles described in Section~\ref{subsec:rational-krylov-poles}.

\begin{remark}
	\label{rem:algorithm-abs-vs-rel}
	In the following, we are going to use $\hat \epsilon$ to denote an absolute error, to distinguish it from the target relative accuracy $\epsilon$. 
	Note that we can easily transform absolute inequalities for the error into relative inequalities if we know in advance an estimate or a lower bound for $\trace(f(A))$. 
	Recall that if $A$ is an $M$-matrix, $\traceP_d(f(A))$ is actually a lower bound for $\trace(f(A))$ (Proposition~\ref{prop:probing-graph-entropy-lower-bound}). 
	In the general case, any rough approximation of the entropy can be used for this purpose, since the important point is determining the order of magnitude of $\trace(f(A))$.
\end{remark}
	
\begin{remark}
	\label{rem:algorithm-error-split}
	The error in the approximation of $S(A)$ can be divided into the error in the approximation of the trace using a probing method or a stochastic estimator, and the error in the approximation of the quadratic forms with $f(A)$ using a Krylov subspace method. For simplicity, in the following we impose that the relative error associated to each of these two components is smaller than $\epsilon/2$.
\end{remark}

\subsection{Probing method implementation}
\label{subsec:probing-method-implementation}

We begin by observing that it is not possible to cheaply estimate the error of a probing method a posteriori, since error estimates are usually based on approximations with different values of the distance $d$, which in general lead to completely different colorings that would require computing all quadratic forms from scratch.

For this reason, it is best to find a value of $d$ that ensures a relative accuracy~$\epsilon$ a priori when using a distance-$d$ coloring. This can be done using one of the bounds in Corollary~\ref{cor:entropy-probing-bound}, but it can often lead to unnecessary additional work, since the bounds usually overestimate the error by a couple of orders of magnitude; see Figure~\ref{fig:traceest-example-probing}. 
Therefore we also provide a heuristic criterion for choosing $d$ that does not have the same theoretical guarantee as the bounds, but appears to work quite well in practice.
In view of Corollary~\ref{cor:entropy-probing-bound}, we can expect the absolute error to behave as
\begin{equation}
	\label{eqn:probing-heuristic-convergence}
	\abs{\trace(f(A)) - \traceP_d(f(A))} \sim \frac{C}{d^{k}} q^d,
\end{equation}
for $k = 2$ and some parameters $C > 0$ and $q \in (0,1)$. However, we found that sometimes the actual error behavior is better described with a different value of $k$, such as $k = 3$, so we do not impose that $k = 2$. 
To estimate the values of the parameters, we compute $\traceP_d(f(A))$ for $d = 1, 2, 3$ and use the estimate
\begin{equation*}
	\abs{\trace(f(A) - \traceP_d(f(A))} \approx \abs{\traceP_{d+1}(f(A)) - \traceP_d(f(A))}, \qquad d = 1, 2.
\end{equation*}
Assuming that~\eqref{eqn:probing-heuristic-convergence} holds exactly and fixing the value of $k$, we can determine $C$ and $q$ by solving the equations
\begin{equation*}
	\abs*{\traceP_{d+1}(f(A)) - \traceP_d(f(A))} = \frac{C}{d^k} q^d, \qquad d = 1, 2.
\end{equation*}
We can check when the resulting estimate is below $\hat \epsilon$ to heuristically determine $d$, i.e., we select~$d$ as
\begin{equation*}
	d_\star = \min \Big\{ d \,:\, \frac{C}{d^k} q^d \le \hat \epsilon \Big\},
\end{equation*}
in order to have the approximate absolute error inequality
\begin{equation*}
	\abs*{\trace(f(A)) - \traceP_{d{_\star}}(f(A))} \lesssim \hat \epsilon.
\end{equation*}
We found that the best results are obtained for $k = 2$ and $k = 3$, so in our implementation we use the maximum of the two corresponding estimates.
Variants of this estimate include using other values of $d$ to estimate the parameters instead of $d = 1, 2, 3$, and using four different values in order to also estimate the parameter~$k$. However, they usually give results that are similar or sometimes worse than the estimate presented above, so they are often not worth the additional effort required to compute them. In particular, using four values of $d$ raises the risk of misjudging the value of $q$, causing the estimate to be inaccurate for large values of $d$.
The error estimate~\eqref{eqn:probing-heuristic-convergence} is compared to the actual error and the theoretical bound~\eqref{eqn:entropy_probing_bound} for two different graphs in Figure~\ref{fig:probing-heuristic-estimate}. The figure also includes a simple error estimate based on consecutive differences, which requires the computation of $\traceP_{d+1}(f(A))$ to estimate the error for $\traceP_d(f(A))$.

\begin{figure}[htbp]
	\makebox[\linewidth][c]{
	\begin{subfigure}[t]{.5\textwidth}
		\includegraphics[width=\textwidth]{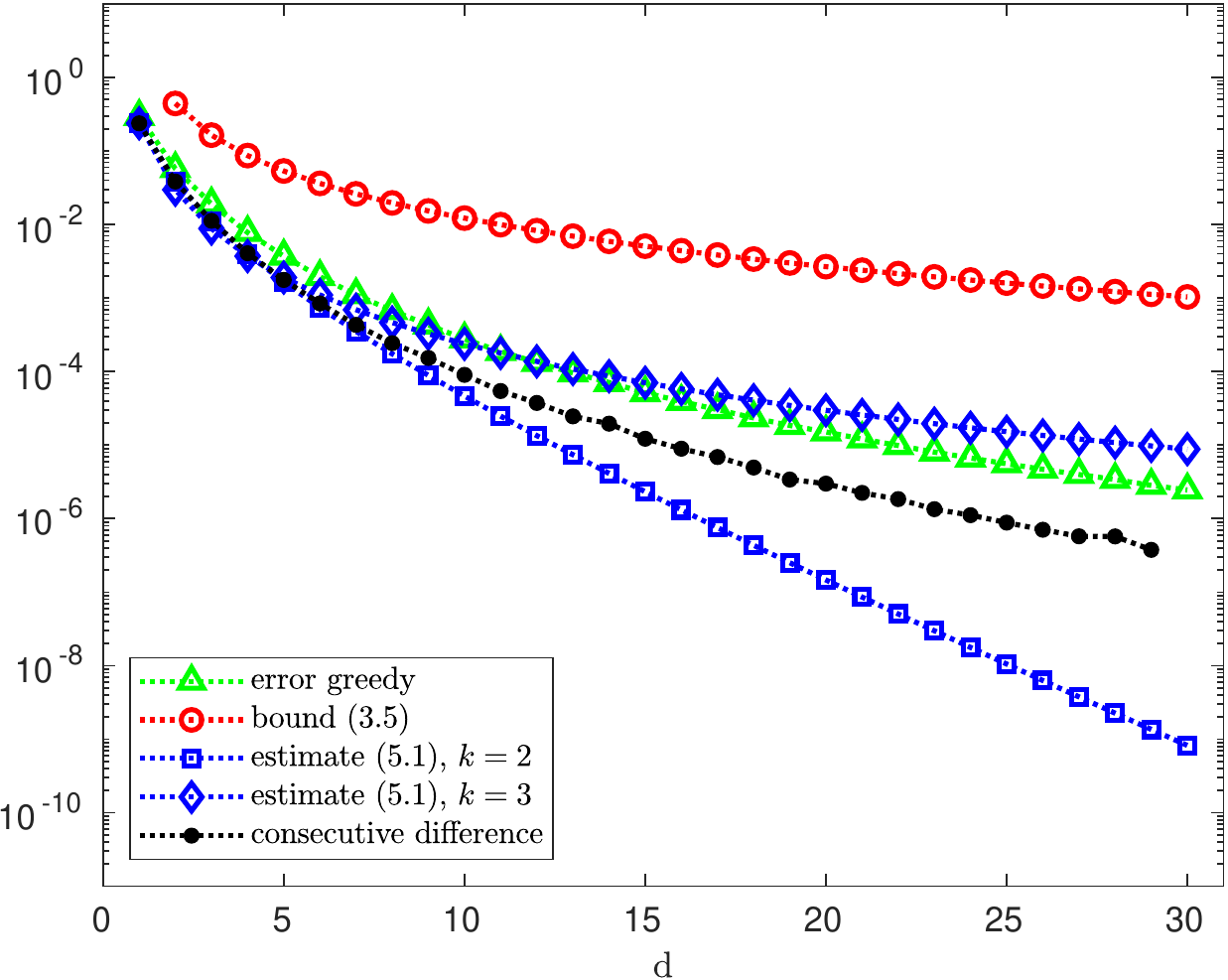}
	\end{subfigure}
	\begin{subfigure}[t]{.5\textwidth}
		\includegraphics[width=\textwidth]{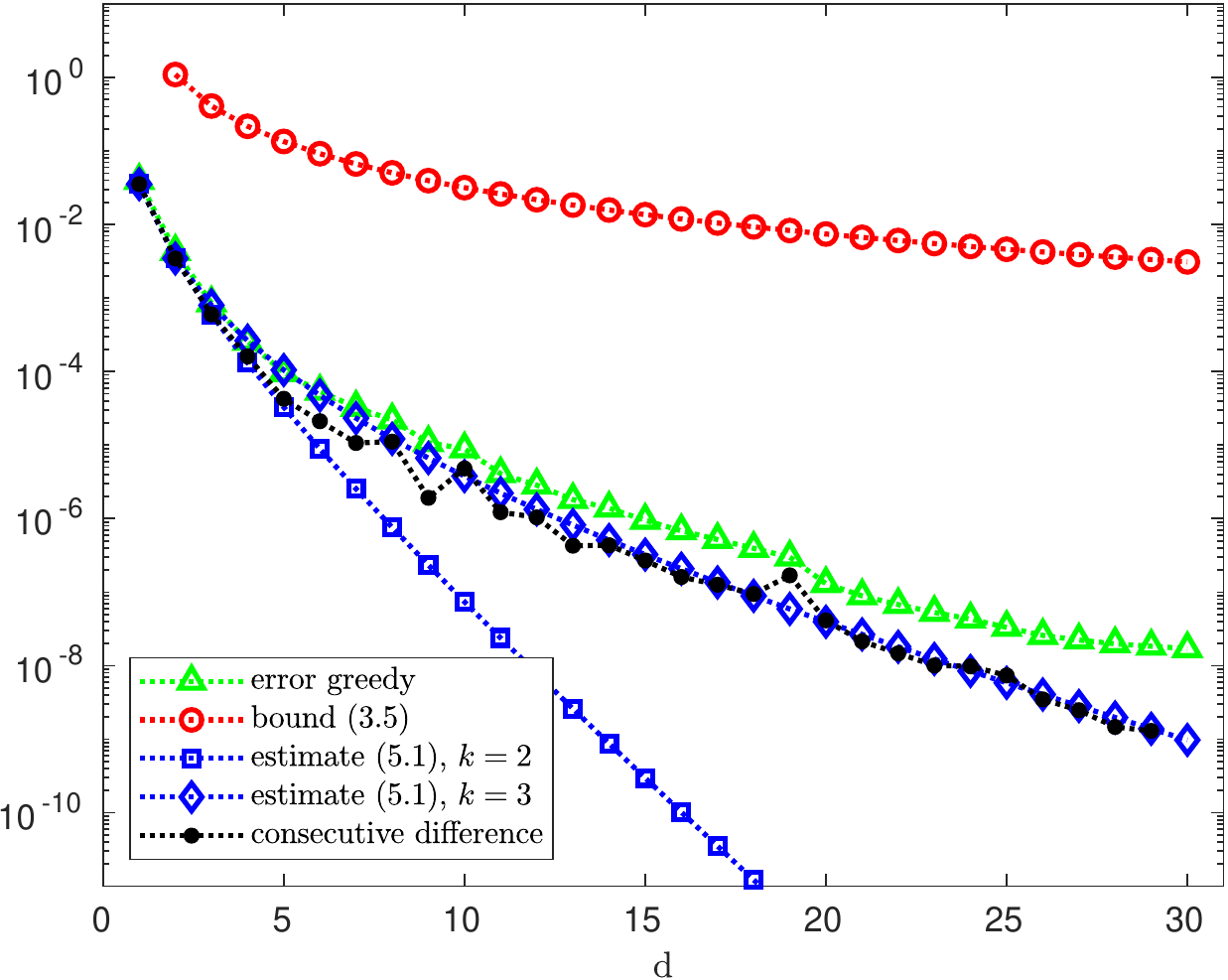}
	\end{subfigure}
	}
	\caption{
	Absolute error of the probing method with greedy coloring (Algorithm~\ref{algorithm:greedy_coloring}) compared with the theoretical bound~\eqref{eqn:entropy_probing_bound} using $a = \lambda_2$, the heuristic error estimate~\eqref{eqn:probing-heuristic-convergence} and the simple error estimate $\abs{\trace(f(A) - \traceP_d(f(A))} \approx \abs{\traceP_{d+1}(f(A)) - \traceP_d(f(A))}$. Left: Laplacian of the largest connected component of the graph \texttt{minnesota}, with $2640$ nodes. Right: Laplacian of the largest connected component of the graph \texttt{eris1176}, with $1174$ nodes.
	\label{fig:probing-heuristic-estimate}}
\end{figure}

\begin{remark}
	\label{rem:probing-estimate-cost}
	The heuristic criterion for selecting $d$ requires the computation of $\traceP_d(f(A))$ for $d = 1, 2, 3$, so it is more expensive to use than the theoretical bound~\eqref{eqn:entropy_probing_bound}. However, this cost is usually small compared to the cost of computing $\traceP_d(f(A))$ for the selected value of $d$, especially if the requested accuracy is small.  Note also that the heuristic criterion always computes $\traceP_3(f(A))$, so it does more work than necessary when $d \le 2$ would be sufficient. Nevertheless, in such a situation the theoretical bound~\eqref{eqn:entropy_probing_bound} may suggest to use an even higher value of $d$ (see Figure~\ref{fig:probing-heuristic-estimate}).
\end{remark}

After choosing $d$ such that 
\begin{equation}
	\label{eqn:algorithm-probing-error-requirement}
	\abs{\trace(f(A)) - \traceP_d(f(A))} \le \hat \epsilon,
\end{equation}
using either the a priori bound~\eqref{eqn:entropy_probing_bound} or the estimate~\eqref{eqn:probing-heuristic-convergence}, a distance-$d$ coloring can be computed with one of the coloring algorithms described in Section~\ref{subsec:probing-methods}, depending on the properties of the graph. The greedy coloring~\cite[Algorithm~4.2]{SchimmelThesis} is usually a good choice for general graphs.

Let us now determine the accuracy required in the computation of the quadratic forms. Assume that we have
\begin{equation*}
	\traceP_d(f(A)) = \sum_{\ell = 1}^s \vec v_\ell^T f(A) \vec v_\ell,
\end{equation*}
where $\{\vec v_\ell\}_{\ell = 1}^s$ are the probing vectors used in the distance-$d$ coloring. 
Let $\hat{\psi}_\ell$ denote the approximation of $\vec v_\ell^T f(A) \vec v_\ell$ obtained with a Krylov method. Recall that $\norm{\vec v_\ell}_2 = \abs{V_\ell}^{1/2}$, where $V_\ell$ denotes the set of the partition associated to the $\ell$-th color.
If we impose the conditions
\begin{equation}
	\label{eqn:algorithm-probing-quadform-error-requirement-abs}
	\abs*{\vec v_\ell^T f(A) \vec v_\ell - \hat{\psi}_\ell} \le \hat \epsilon \cdot \frac{\abs{V_\ell}}{n} \qquad \ell = 1, \dots, s,
\end{equation}
where we normalized the accuracy requested for each quadratic form depending on $\norm{\vec v_\ell}_2$, we obtain the desired absolute accuracy on the probing approximation
\begin{equation}
	\label{eqn:algorithm-probing-krylov-error-requirement}
	\Big\lvert \traceP_d(f(A)) - \sum_{\ell = 1}^s \hat{\psi}_\ell\Big\rvert \le \hat \epsilon.
\end{equation}
If we are aiming for a relative accuracy $\epsilon$, we should select $\hat \epsilon = \frac{1}{2}\,\epsilon \trace (f(A))$ in \eqref{eqn:algorithm-probing-error-requirement} and \eqref{eqn:algorithm-probing-krylov-error-requirement}. 
In practice, $\trace(f(A))$ will be replaced by a rough approximation (see Remark~\ref{rem:algorithm-abs-vs-rel}).

The overall probing algorithm is summarized in Algorithm~\ref{algorithm:probing-entropy}.

\begin{algorithm}
    \caption{Probing method for $S(A)$} \label{algorithm:probing-entropy}
    \begin{algorithmic}[1]
    \Require $A \in \R^{n \times n}$ density matrix, $\epsilon$ relative error tolerance
    \Ensure $\text{trest} \approx S(A)$ such that $\abs{\text{tr} - S(A)} / S(A) \lesssim \epsilon$

    \State Select $d$ such that $\abs{\traceP_d(f(A)) - S(A)}/S(A) \lesssim \epsilon/2$, using either the bound \eqref{eqn:entropy_probing_bound} or the heuristic \eqref{eqn:probing-heuristic-convergence}. The heuristic \eqref{eqn:probing-heuristic-convergence} requires the computation of $\traceP_d(f(A))$ for $d = 1, 2, 3$, which can be done by running steps 2 -- 4.
    \State Compute a distance-$d$ coloring of $G(A)$ with, e.g., Algorithm~\ref{algorithm:greedy_coloring} and the associated probing vectors $\{\vec v_1, \dots, \vec v_s\}$.
    \State For $\ell = 1, \dots, s$, compute $\hat \psi_\ell \approx \vec v_\ell^T f(A) \vec v_\ell$ such that \eqref{eqn:algorithm-probing-quadform-error-requirement-abs} holds, using a rational Krylov method with either the upper bound~\eqref{eqn:error-bound-a-posteriori} or the estimate~\eqref{eqn:error-bound-a-posteriori--estimate} as stopping criterion.
	\State \Return $\text{trest} = \dsum_{\ell = 1}^s \hat \psi_\ell$, satisfying $\Big \lvert \dsum_{\ell = 1}^s \hat \psi_\ell - S(A) \Big \rvert / S(A) \lesssim \epsilon$.
    \end{algorithmic}
\end{algorithm}

\subsection{Adaptive Hutch++ implementation} 
\label{subsec:aHutchpp_implementation}

We use the Matlab code of~\cite[Algorithm~3]{PerssonCortinovisKressner21} provided by the authors, modified to use Krylov methods for the computations with $f(A)$. This algorithm requires an absolute tolerance $\hat\epsilon$ and a failure probability $\delta$, and outputs an approximation $\trace_\text{adap}(f(A))$ such that
\begin{equation*}
	\prob\big[\abs{\trace(f(A)) - \trace_{\text{adap}} (f(A))} \ge \hat\epsilon \big] \leq \delta.
\end{equation*}
To obtain an approximation within a relative accuracy $\epsilon$, we can use $\hat\epsilon \approx \epsilon \trace(f(A))$, using a rough approximation of $\trace(f(A))$. 
Similarly to the probing method, in order to have a final relative error bounded by $\epsilon$, in our implementation we use a tolerance $\hat \epsilon \approx \frac{1}{2} \epsilon \trace (f(A))$ for adaptive Hutch++, and we set the accuracy for the computation of matrix-vector products and quadratic forms in order to ensure that the total error due to the Krylov approximations remains below~$\frac{1}{2} \epsilon \trace (f(A))$. We omit the technical details to simplify the presentation.

\subsection{Krylov method implementation}

Quadratic forms with $f(A)$ are approximated using a Krylov method with some poles at $\infty$ followed by the EDS poles of Section~\ref{subsec:rational-krylov-poles}, using as a stopping criterion either the a posteriori upper bound~\eqref{eqn:error-bound-a-posteriori} or the estimate shown in Example~\ref{example:krylov-bound}. 

The number of poles at $\infty$ is chosen in an adaptive way, switching to finite poles when the error reduction in the last few iterations of the polynomial Krylov method is ``small''. Specifically, we decide to switch to EDS poles after the $k$-th iteration if on average the last $\ell \ge 1$ iterations did not reduce the error bound or estimate \texttt{err\_est} by at least a factor $c \in (0, 1)$, i.e.~if
\begin{equation*}
	\frac{\texttt{err\_est}_{k}}{\texttt{err\_est}_{k-\ell-1}} \ge c^\ell.
\end{equation*}
In our implementation we use $\ell = 3$ and $c = 0.75$, usually leading to at most~$10$ polynomial Krylov iterations.

Since EDS poles are contained in $(-\infty, 0)$, each rational Krylov iteration involves the solution of a symmetric positive definite linear system, which can be computed either with a direct method using a sparse Cholesky factorization, or iteratively with the conjugate gradient method using a suitable preconditioner.
Note that the same EDS poles can be used for all quadratic forms, so the number of different matrices that appear in the linear systems is usually small and independent of the total number of quadratic forms. 
Although this depends on the accuracy requested for the entropy, the number of EDS poles used is almost always bounded by $10$, and often much smaller than that: see the numerical experiments in Section~\ref{sec:numerical-experiments} for some examples. 
This is a great advantage for direct methods, especially when the Cholesky factor remains sparse, since we can compute and store a Cholesky factorization for each pole and then reuse it for all quadratic forms. If the fill-in in the Cholesky factor is moderate, the cost of a rational iteration can become comparable to the cost of a polynomial one, leading to large savings when computing many quadratic forms. 
Of course, for large matrices with a general sparsity structure the computation of even a single Cholesky factor may be unfeasible, so the only option is to use a preconditioned iterative method. In such a situation, it is still possible to benefit from the small number of different matrices that appear in linear systems by storing and reusing preconditioners, but the gain is less evident compared to direct methods.

The matrix-vector products with $f(A)$ in the Hutch++ algorithm are approximated with the same Krylov subspace method, with the difference that we use the a posteriori upper and lower bounds from~\cite[eq.~(6.15)]{GuettelThesis}. A geometric mean estimate similar to the one used in Example~\ref{example:krylov-bound} can be also used in this context.
For the computation of the a posteriori bounds we use the pole swapping technique with an auxiliary pole at $\infty$ described in Section~\ref{subsubsec:krylov-bound-computation}.

\section{Numerical experiments}
\label{sec:numerical-experiments}

The experiments were done in Matlab R2021b on a laptop with operating system Ubuntu 20.04, using a single core of an Intel i5-10300H CPU running at 2.5 GHz, with 32 GB of RAM. 
Since we are using Matlab, the execution times may not reflect the performance of a high performance implementation, but they are still a useful indicator when comparing different methods.

\subsection{Test matrices}

We consider a number of symmetric test matrices from the SuiteSparse Matrix Collection~\cite{DavisSparseCollection}. All matrices are treated as binary matrices, i.e.~all edge weights are set to one. For each matrix, we extract the graph Laplacian associated to the largest connected component and we normalize it so that it has unit trace.  We report some information on the resulting matrices in Table~\ref{table:matrixdata}. For the four smallest matrices, the eigenvalues were computed via diagonalization, while for the larger matrices the eigenvalues $\lambda_2$ and $\lambda_n$ were approximated using \texttt{eigs}. The cost of solving a linear system with a direct method is highly dependent on the fill-in in the Cholesky factorization; the column labelled \texttt{fill-in} in Table~\ref{table:matrixdata} contains the ratios $\texttt{nnz}(R)/\texttt{nnz}(\rho)$, where $\rho$ is the test matrix and $R$ is the Cholesky factor of any shifted matrix $\rho + \alpha I$, for $\alpha > 0$ (\texttt{nnz}$(M)$ denotes the number of nonzeros of a matrix~$M$).
All matrices have been ordered using the approximate minimum degree reordering option available in Matlab before factorization.

\begin{table}[ht]
	\centering
	\caption{Information on the matrices used in the experiments.}
	\label{table:matrixdata}
	\small
	\begin{tabular}{l rrrrrr}
		\toprule
		test matrix & $n$ & \texttt{nnz}$(\rho)$ & \texttt{fill-in} & $\lambda_2$ & $\lambda_n$ & entropy \\
		\midrule
		\texttt{yeast} & 2224 & 15442 & 3.6 & 4.54e-06 & 4.96e-03 & 7.055 \\ 
		\addlinespace[1mm] 
		\texttt{minnesota} & 2640 & 9244 & 1.3 & 1.28e-07 & 1.04e-03 & 7.607 \\ 
		\addlinespace[1mm] 
		\texttt{ca-HepTh} & 8638 & 58250 & 7.5 & 4.92e-07 & 1.33e-03 & 8.540 \\ 
		\addlinespace[1mm] 
		\texttt{bcsstk29} & 13830 & 618678 & 2.9 & 7.22e-08 & 1.25e-04 & 9.440 \\ 
		\addlinespace[1mm] 
		\texttt{cond-mat-2005} & 36458 & 379926 & 21.7 & 5.63e-08 & 8.13e-04 & 9.958 \\ 
		\addlinespace[1mm] 
		\texttt{loc-Brightkite} & 56739 & 482629 & 32.6 & 7.10e-08 & 2.67e-03 & 9.896 \\ 
		\addlinespace[1mm] 
		\texttt{ut2010} & 115406 & 687472 & 1.2 & 2.72e-10 & 3.44e-04 & 11.361 \\ 
		\addlinespace[1mm] 
		\texttt{usroads} & 126146 & 450046 & 1.4 & 2.39e-11 & 2.54e-05 & 11.478 \\ 
		\addlinespace[1mm] 
		\texttt{com-Amazon} & 334863 & 2186607 & 105.9 & 6.69e-10 & 2.97e-04 & 12.400 \\ 
		\addlinespace[1mm] 
		\texttt{ny2010} & 350167 & 2059711 & 1.8 & 4.67e-12 & 3.63e-05 & 12.541 \\ 
		\addlinespace[1mm] 
		\texttt{roadNet-PA} & 1087562 & 4170590 & 1.6 & 5.54e-13 & 3.37e-06 & 13.628 \\ 
		\addlinespace[1mm] 
\bottomrule
	\end{tabular}
\end{table}

\subsection{Probing bound vs. estimate}
\label{subsec:experiments--probing-bound-vs-estimate}

In this experiment, we fix a relative error tolerance $\epsilon = 10^{-3}$ and we compare the choice of $d$ given by the theoretical bound~\eqref{eqn:entropy_probing_bound} with the one provided by the heuristic estimate~\eqref{eqn:probing-heuristic-convergence}. We report in Table~\ref{table:probing-bound-vs-estimate} the error, the execution time, the value of $d$ and the number of colors used in the two cases. 
When the theoretical bound is used, the selected value of $d$ is significantly higher compared to the one chosen by the heuristic estimate, but in both cases the overall error remains below the tolerance $\epsilon$. Moreover, for certain graphs using the theoretical bound leads to greedy colorings with a number of colors equal to the number of nodes in the graph, completely negating the advantage of using a probing method.
Observe that the errors obtained with the larger value of $d$ are not much smaller than the ones for the smaller value of $d$, because in both cases the quadratic forms are computed with target relative accuracy $\epsilon$, so the probing error for the larger value of $d$ is dominated by the error in the quadratic forms. 
In the case of the heuristic estimate, the execution time includes the time required to run the probing method for $d = 1, 2, 3$ in order to evaluate~\eqref{eqn:probing-heuristic-convergence}.
The stopping criterion for the Krylov subspace method uses the upper bound~\eqref{eqn:error-bound-a-posteriori}. The execution time can be further reduced by using the estimate~\eqref{eqn:error-bound-a-posteriori--estimate} for the Krylov subspace method, as the following experiment shows. The diagonalization time for the matrices used in this experiment can be found in the last column of Table~\ref{table:hutchpp-experiment-1}. Note that for smaller matrices, diagonalization is often the fastest method, but the advantage of approximating the entropy with a probing method is already evident for matrices of size $n \approx 10000$.

\begin{table}[ht]
	\centering
	\caption{Comparison of the theoretical bound~\eqref{eqn:entropy_probing_bound} against the heuristic estimate~\eqref{eqn:probing-heuristic-convergence} for choosing $d$, using relative tolerance $\epsilon = 10^{-3}$ in the probing method. Top row: heuristic estimate. Bottom row: theoretical bound.}
	\label{table:probing-bound-vs-estimate}
	\small
	\begin{tabular}{l rr crrr}
		\toprule
		test matrix & $n$ & & error & $d$ & colors & time (s) \\
		\midrule
		\multirow{2}{*}{\texttt{yeast}} 
		& \multirow{2}{*}{2224} 
		&& 3.062e-04 & 3 & 222 & 0.952 \\ 
		&&& 3.733e-05 & 25 & 2224 & 4.464 \\ 
		\addlinespace[1mm] 
		\multirow{2}{*}{\texttt{minnesota}} 
		& \multirow{2}{*}{2640} 
		&& 4.456e-04 & 5 & 24 & 0.196 \\ 
		&&& 3.173e-05 & 18 & 255 & 0.779 \\ 
		\addlinespace[1mm] 
		\multirow{2}{*}{\texttt{ca-HepTh}} 
		& \multirow{2}{*}{8638} 
		&& 2.974e-04 & 3 & 252 & 2.273 \\ 
		&&& 3.161e-05 & 27 & 8638 & 42.078 \\ 
		\addlinespace[1mm] 
		\multirow{2}{*}{\texttt{bcsstk29}} 
		& \multirow{2}{*}{13830} 
		&& 4.912e-05 & 3 & 176 & 2.292 \\ 
		&&& 8.497e-05 & 12 & 2095 & 17.849 \\ 
		\addlinespace[1mm] 
				\bottomrule
	\end{tabular}
\end{table}

\subsection{Krylov bound vs. estimate}
We fix an error tolerance $\epsilon = 10^{-5}$ and compare the performance of the geometric mean error estimate~\eqref{eqn:error-bound-a-posteriori--estimate} with the theoretical upper bound~\eqref{eqn:error-bound-a-posteriori} for the Krylov subspace method. 
The value of $d$ for the probing method is selected using the heuristic estimate~\eqref{eqn:probing-heuristic-convergence}. The entropy error, execution time, and total number of polynomial and rational Krylov iterations are reported in Table~\ref{table:krylov-bound-vs-estimate}. We can see that using the estimate instead of the theoretical bound moderately reduces the computational effort, while still attaining the requested accuracy~$\epsilon$ on the entropy. In particular, observe that the number of rational Krylov iterations is significantly higher when using the upper bound~\eqref{eqn:error-bound-a-posteriori}. In this experiment, all linear systems are solved with direct methods and Cholesky factorizations are stored and reused.

\begin{table}[ht]
	\centering
	\caption{Comparison of the upper bound~\eqref{eqn:error-bound-a-posteriori} against the geometric mean estimate~\eqref{eqn:error-bound-a-posteriori--estimate} for Krylov methods used in the probing method, using relative tolerance $\epsilon = 10^{-5}$ for the probing method. Top row: geometric mean estimate. Bottom row: upper bound.}
	\label{table:krylov-bound-vs-estimate}
	\small
	\begin{tabular}{l rr crrr}
		\toprule
		test matrix & $n$ & & error & poly iter & rat iter & time (s)  \\
		\midrule
		\multirow{2}{*}{\texttt{yeast}} 
		& \multirow{2}{*}{2224} 
		&& 4.405e-06 & 16140 & 748 & 7.095 \\ 
		&&& 6.023e-06 & 16419 & 2237 & 8.231 \\ 
		\addlinespace[1mm] 
		\multirow{2}{*}{\texttt{minnesota}} 
		& \multirow{2}{*}{2640} 
		&& 5.728e-07 & 2983 & 289 & 1.535 \\ 
		&&& 2.490e-06 & 2987 & 598 & 1.734 \\ 
		\addlinespace[1mm] 
		\multirow{2}{*}{\texttt{ca-HepTh}} 
		& \multirow{2}{*}{8638} 
		&& 8.195e-07 & 39481 & 2442 & 40.240 \\ 
		&&& 2.395e-06 & 39747 & 6389 & 50.133 \\ 
		\addlinespace[1mm] 
		\multirow{2}{*}{\texttt{bcsstk29}} 
		& \multirow{2}{*}{13830} 
		&& 6.133e-06 & 4704 & 0 & 12.093 \\ 
		&&& 7.545e-06 & 6363 & 44 & 16.047 \\ 
		\bottomrule		
	\end{tabular}
\end{table}

\subsection{Adaptive Hutch++}
\label{subsec:experiments--adaptive-hutchpp}

Here we test the performance and the accuracy of the adaptive implementation of Hutch++~\cite[Algorithm~3]{PerssonCortinovisKressner21}. A relative accuracy $\epsilon$ is achieved by setting the absolute tolerance to $\epsilon S(\rho)$, where $S(\rho)$ is computed via diagonalization and considered as exact. The computational effort of Hutch++ is determined by the parameters $N_r$ and $N_H$ described in Section \ref{subsec:stochastic-trace-estimation}. In particular, the number of matrix-vector products is equal to $N_r$ and the number of quadratic forms is equal to $N_r+N_H$. 
We used $\epsilon=10^{-2}, 10^{-3}$ as target tolerances and $\delta=10^{-2}$ as failure probability. 
Matrix-vector products and quadratic forms are computed using the Krylov subspace method with the geometric mean estimate as a stopping criterion \eqref{eqn:error-bound-a-posteriori--estimate}.
In Table \ref{table:hutchpp-experiment-1} we compare the results for the two tolerances, obtained as an average of $100$ runs of the algorithm, including both the average and worst relative error. In the majority of cases, the worst error is below the input tolerance $\epsilon$.
We see that for $\epsilon = 10^{-2}$ the computation with Hutch++ is very fast for all test matrices; on the other hand, the cost becomes significantly higher for $\epsilon = 10^{-3}$, showing that the stochastic estimator quickly becomes inefficient as the required accuracy increases. Observe that for $\epsilon = 10^{-2}$ adaptive Hutch++ uses only $3$ matvecs for all tests problems, which is the minimum amount that can be used by the implementation in~\cite{PerssonCortinovisKressner21}. This means that the internal criteria of the algorithm have determined that spending more matvecs in the low rank approximation is not beneficial, and hence the convergence of the method is roughly the same as for Hutchinson's estimator. A similar behavior can be observed in Table~\ref{table:hutchpp-experiment-2}, and can be linked to the fact that for these test matrices $\rho$, the matrix function $-\rho \log \rho$ does not exhibit eigenvalue decay and hence cannot be well-approximated by low rank matrices. On the other hand, for problems where low rank approximation is more effective, stochastic estimators that exploit it such as Hutch++ can have much faster convergence.

\begin{table}[ht]
	\centering
	\caption{Results for Hutch++ applied to some test matrices. For each matrix, the first and second row show the results for $\epsilon=10^{-2}$ and $\epsilon=10^{-3}$, respectively. The failure probability is $\delta=10^{-2}$ in both cases. The last column contains the diagonalization times. 
	}
	\label{table:hutchpp-experiment-1}
	\small 
	\begin{tabular}{l rrr rrr}
		\toprule
		test matrix & avg error & worst error & $N_r$ & $N_r + N_H$ & time (s) & \texttt{eig} (s)\\
		\midrule
		\multirow{2}{*}{\texttt{yeast}} 
		& 2.55e-03 & 9.94e-03 & 3 & 282 & 0.421  
		& \multirow{2}{*}{0.508} \\ 
		& 3.56e-04 & 1.16e-03 & 1228 & 2160 & 19.735  \\ 
		\addlinespace[1mm] 
		\multirow{2}{*}{\texttt{minnesota}} 
		& 3.46e-03 & 1.07e-02 & 3 & 154 & 0.111  
		& \multirow{2}{*}{0.819} \\ 
		& 4.53e-04 & 1.26e-03 & 1854 & 2684 & 35.721  \\ 
		\addlinespace[1mm] 
		\multirow{2}{*}{\texttt{ca-HepTh}} 
		& 2.83e-03 & 9.92e-03 & 3 & 81 & 0.205  
		& \multirow{2}{*}{22.170} \\ 
		& 3.55e-04 & 9.14e-04 & 635 & 3968 & 66.350  \\ 
		\addlinespace[1mm] 
		\multirow{2}{*}{\texttt{bcsstk29}} 
		& 1.84e-03 & 6.97e-03 & 3 & 38 & 0.171  
		& \multirow{2}{*}{86.696} \\ 
		& 2.39e-04 & 8.24e-04 & 3 & 1883 & 13.276  \\ 
		\addlinespace[1mm] 
		\bottomrule
	\end{tabular}

\end{table}
\subsection{Larger matrices}
In this section we test the probing method and the adaptive Hutch++ algorithm on larger matrices, for which it would be extremely expensive to compute the exact entropy. In light of the results shown in Tables~\ref{table:probing-bound-vs-estimate} and~\ref{table:krylov-bound-vs-estimate}, we select the value of $d$ for the probing method using the heuristic estimate~\eqref{eqn:probing-heuristic-convergence} and we use the geometric mean estimate~\eqref{eqn:error-bound-a-posteriori--estimate} for the Krylov subspace method. 
The results are reported in Tables~\ref{table:probing-experiment-large-1} and~\ref{table:probing-experiment-2} for the probing method, and in Table~\ref{table:hutchpp-experiment-2} for Hutch++. Figure~\ref{fig:probing-experiment-times} contains a more detailed breakdown of the execution time for the probing method used on the matrices of Table~\ref{table:probing-experiment-large-1}. 
We separate the time in preprocessing, where we evaluate the heuristic~\eqref{eqn:probing-heuristic-convergence} to select $d$, and the main run of the algorithm with the chosen value of $d$. The time for the main run is further divided in coloring, and polynomial and rational Krylov iterations. The time for the Cholesky factorizations refers to the whole process, since the factors are computed and stored when a certain pole for a rational Krylov iteration is encountered for the first time.

In Table~\ref{table:probing-experiment-large-1}, we consider matrices with a ``large-world'' sparsity structure, such as road networks, and we use a relative tolerance of $\epsilon = 10^{-4}$. For these matrices, for which the diameter and the average path length are relatively large,
it is possible to compute distance-$d$ colorings with a relatively small number of colors, and therefore probing methods converge quickly. Moreover, Cholesky factorizations can be computed cheaply and have a small fill-in, so it is possible to rapidly solve linear systems using a direct method.
On the other hand, in Table~\ref{table:probing-experiment-2} we consider matrices with a ``small-world'' sparsity structure, more typical of social networks and scientific collaboration networks. These matrices require a much larger number of colors to construct distance-$d$ colorings, even for small values of $d$. The cost of probing methods is thus significantly higher on this kind of problem.
In Table~\ref{table:probing-experiment-2}, only polynomial Krylov iterations are used due to the low relative tolerance $\epsilon = 10^{-2}$, so it is never necessary to solve linear systems. Recall that these matrices also have a high fill-in in the Cholesky factorizations (see Table~\ref{table:matrixdata}), so the conjugate gradient method with a suitable preconditioner is likely to be much more efficient than a direct method for solving a linear system.

\begin{table}[ht]
	\centering
	\caption{Results for the probing method applied to test matrices with large-world sparsity structure, using relative tolerance $\epsilon = 10^{-4}$.}
	\label{table:probing-experiment-large-1}
	\small
	\begin{tabular}{l rrrr rrr}
		\toprule
		test matrix & $n$ && $d$ & colors & poly iter & rat iter & time (s) \\
		\midrule
		\texttt{ut2010} & $115406$ && 
		4 & 504 & 7070 & 919 & 79.60 \\ 
		\addlinespace[1mm] 
		\texttt{usroads} & $126146$ && 
		8 & 77 & 626 & 0 & 6.30 \\ 
		\addlinespace[1mm] 
		\texttt{ny2010} & $350167$ && 
		5 & 329 & 3914 & 15 & 111.39 \\ 
		\addlinespace[1mm] 
		\texttt{roadNet-PA} & $1087562$ && 
		8 & 106 & 827 & 0 & 84.46 \\ 
		\bottomrule
	\end{tabular}
\end{table}

\begin{figure}[htbp]
	\centering
	\includegraphics[width=0.65\textwidth]{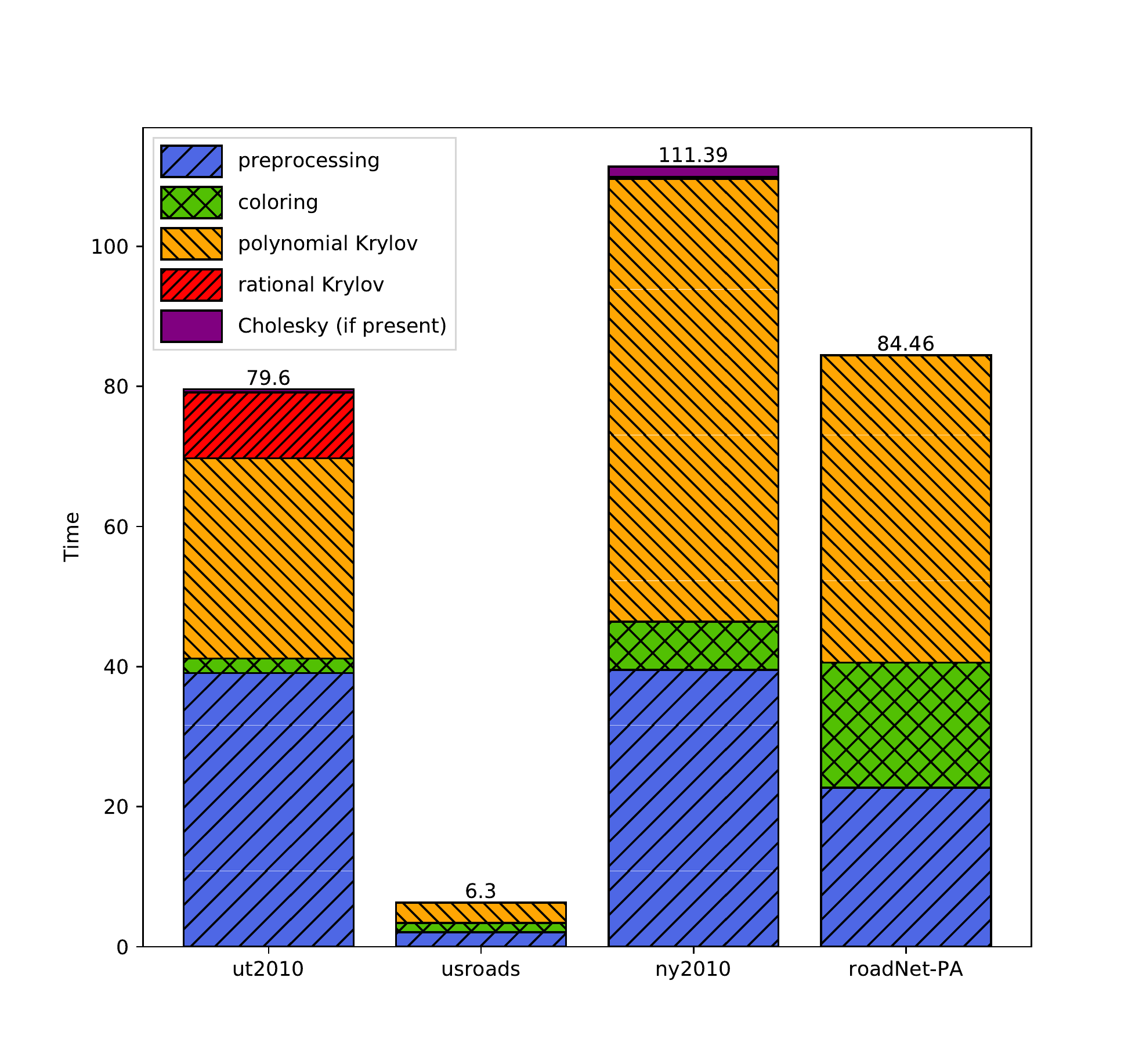}
	\caption{Breakdown of the execution time of the probing method for the test matrices in Table~\ref{table:probing-experiment-large-1}.}
	\label{fig:probing-experiment-times}
\end{figure}

\begin{table}[ht]
	\centering
	\caption{Results for the probing method applied to test matrices with small-world sparsity structure, using relative tolerance $\epsilon = 10^{-2}$.}
	\label{table:probing-experiment-2}
	\small
	\begin{tabular}{l rrr rrr}
		\toprule
		test matrix & $n$ && $d$ & colors & poly iter & time (s) \\
		\midrule
		\texttt{cond-mat-2005} & $36458$ && 
		3 & 1221 & 3883 & 13.809 \\ 
		\addlinespace[1mm] 
		\texttt{loc-Brightkite} & $56739$ && 
		3 & 3946 & 18765 & 90.564 \\ 
		\addlinespace[1mm] 
		\texttt{com-Amazon} & $334863$ && 
		3 & 625 & 1285 & 47.458 \\ 		
		\bottomrule
	\end{tabular}
\end{table}

In Table~\ref{table:hutchpp-experiment-2} we show the results for the adaptive Hutch++ algorithm, using relative tolerance $\epsilon = 10^{-2}$. The results are obtained as an average of $100$ runs of the algorithm. We can observe that the stochastic trace estimator works well for both large-world and small-world graphs, in contrast to the probing method.

\begin{table}[ht]
	\centering
	\caption{Results for Hutch++ applied to large test matrices, with relative tolerance $\epsilon=10^{-2}$ and failure probability $\delta=10^{-2}$. The parameters $N_r$ and $N_H$ are defined in Section~\ref{subsec:stochastic-trace-estimation}.}
	\label{table:hutchpp-experiment-2}
	\small
	\begin{tabular}{l rrr rrr}
		\toprule
		test matrix  & $n$ & $N_r$ & $N_r + N_H$ & time (s) \\
		\midrule
		\texttt{ny2010}  & 350167  & 3 & 10 & 0.490 \\ 
		\addlinespace[1mm] 
		 \texttt{usroads}  & 126146  & 3 & 14 & 0.190 \\ 
		\addlinespace[1mm] 
		 \texttt{ny2010}  & 350167  & 3 & 10 & 0.487 \\ 
		\addlinespace[1mm] 
		 \texttt{roadNet-PA}  & 1087562  & 3 & 8 & 1.126 \\ 
		\midrule
		 \texttt{cond-mat-2005}  & 36458  & 3 & 34 & 0.448 \\ 
		\addlinespace[1mm] 
		 \texttt{loc-Brightkite}  & 56739  & 3 & 42 & 4.576 \\ 
		\addlinespace[1mm] 
		 \texttt{com-Amazon}  & 334863  & 3 & 11 & 1.171 \\ 
		 \addlinespace[1mm] 
		\bottomrule		
	\end{tabular}

\end{table}

\subsection{Algorithm scaling}

To investigate how the complexity of the algorithms scales with the matrix size, we compare the scaling of the probing method and the stochastic trace estimator on two different test problems with increasing dimension. 
The first one is the graph Laplacian of a 2D regular square grid, and the second one is the graph Laplacian of a Barabasi-Albert random graph, generated using the \texttt{pref} function of the CONTEST Matlab package~\cite{CONTEST}. For the probing method on the 2D grid, we use the optimal distance-$d$ coloring with $\left\lceil \frac{1}{2} (d+1)^2 \right\rceil$ colors described in~\cite{FGR03}.
For both test problems, we use a relative tolerance $\epsilon = 10^{-4}$ for the probing method, and a relative tolerance $\epsilon = 10^{-2}$ and failure probability $\delta = 10^{-2}$ for adaptive Hutch++, averaging over $100$ runs. The results are summarized in Figure~\ref{fig:scaling}, for graphs with a number of nodes from $n = 2^{10}$ to $n = 2^{20}$.
As expected, the probing method is much more efficient in the case of the 2D grid, since the number of colors used in the distance-$d$ colorings remains constant as $n$ increases. On the other hand, for the Barabasi-Albert random graph, which has a small-world structure, the number of colors used in a distance-$d$ coloring increases with the number of nodes, and hence the scaling for the probing method is significantly worse. 
The adaptive Hutch++ algorithm also has a better performance for the 2D grid, but the scaling in the problem size is good for both graph categories, since the number of vectors used in the trace approximation does not increase with the matrix dimension. However, the stochastic approach is only viable with a loose tolerance $\epsilon$ for this kind of problem since low rank approximation is not effective, as discussed in Section~\ref{subsec:experiments--adaptive-hutchpp}. The initial decrease in the execution time for Hutch++ as $n$ increases is caused by the fact that the adaptive algorithm uses a larger number of vectors for the graphs with fewer nodes.

\begin{figure}[htbp]
	\makebox[\linewidth][c]{
	\begin{subfigure}[t]{.50\textwidth}
		\includegraphics[width=\textwidth]{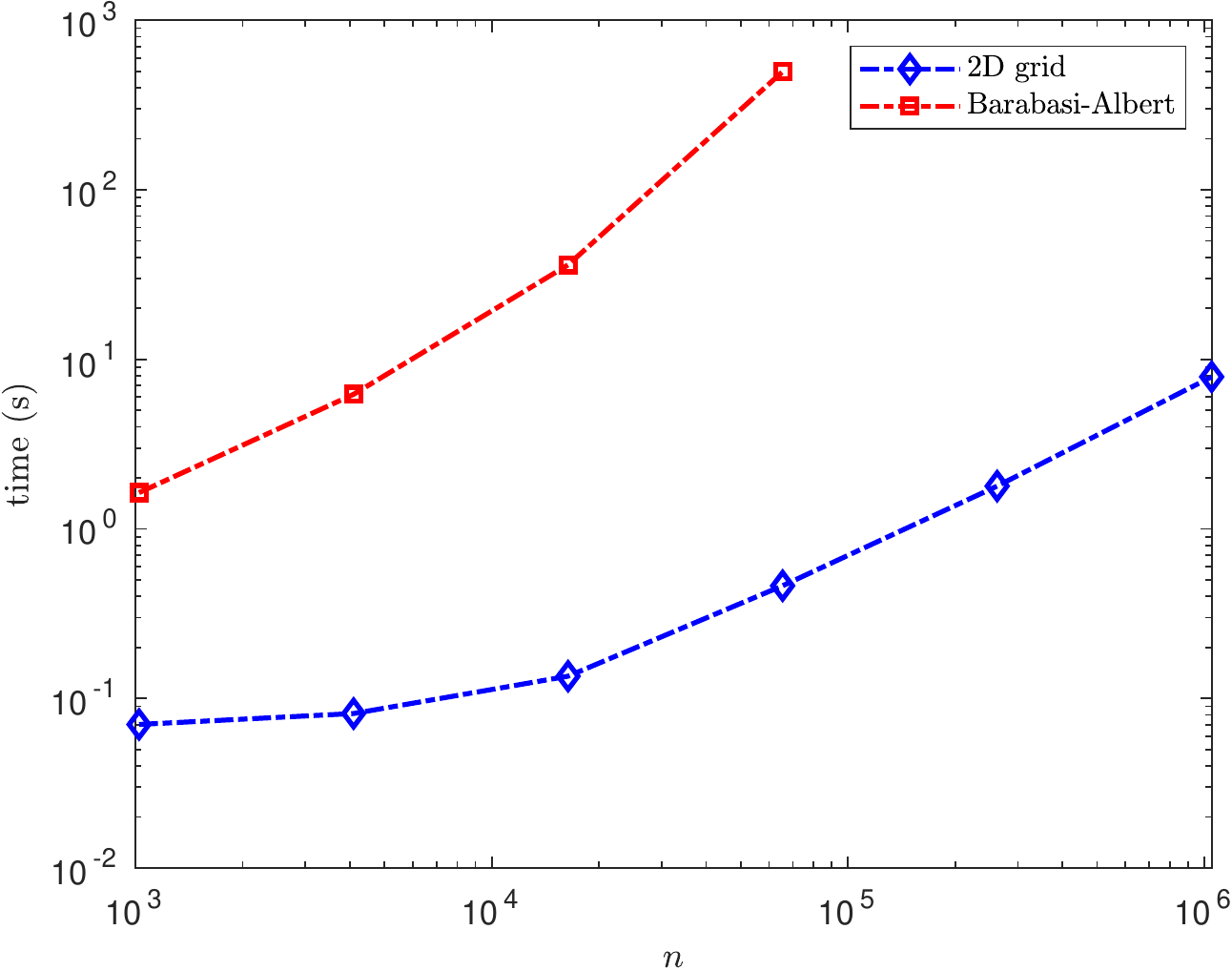}
	\end{subfigure}
	\begin{subfigure}[t]{.50\textwidth}
		\includegraphics[width=\textwidth]{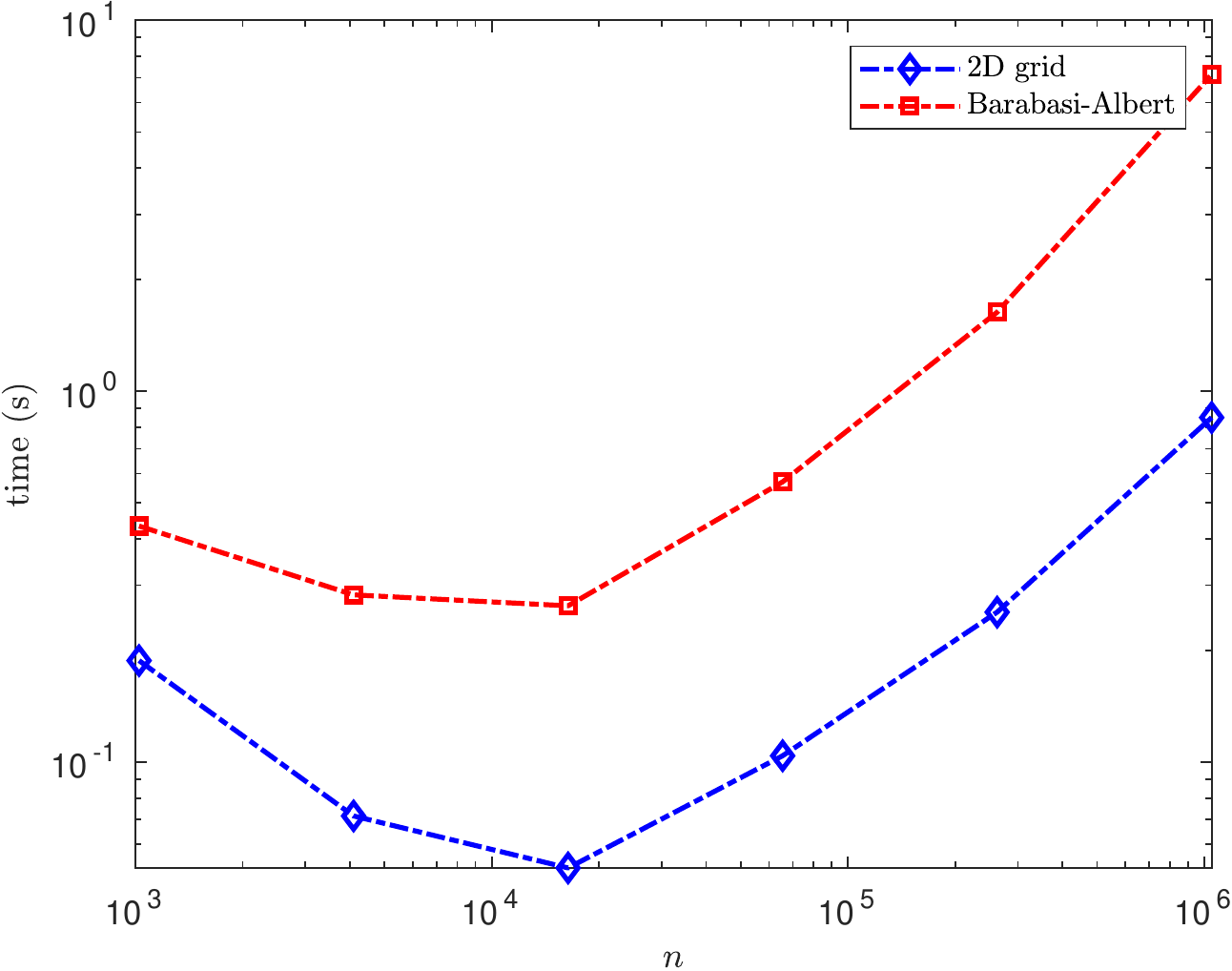}
	\end{subfigure}
	}
	\caption{Execution times for the probing method (left, $\epsilon = 10^{-4}$) and the adaptive Hutch++ algorithm (right, $\epsilon = 10^{-2}$) on the graph Laplacian of a 2D regular grid and a Barabasi-Albert random graph, as a function of the number of nodes $n$. }
	\label{fig:scaling}
\end{figure}

\section{Conclusions}
\label{sec:conclusions}

In this paper we have investigated two approaches for approximating the von Neumann entropy of a large, sparse, symmetric
positive semidefinite matrix.  The first method is a state-of-the-art randomized approach, while the second one is based on 
the idea of probing.
Both methods require the computation of many quadratic forms involving the matrix function $f(A)$
with $f(x) = - x \log x$, an expensive task given the lack of smoothness of $f(x)$ at $x=0$. We have examined the use of 
both polynomial and rational Krylov subspace methods, and combinations of the two. Pole selection and several implementation
aspects, such as heuristics and stopping criteria, have been investigated. Numerical experiments in which the entropy is computed for a variety of
networks have been used to test the various approximation methods. Not surprisingly, the performance of the methods is 
affected by the structure of the underlying network, especially for the method based on the probing idea. Our main conclusion is that the probing approach is better suited than the randomized 
one for graphs with a large-world structure, since they admit distance-$d$ colorings with a relatively small number of colors. Conversely, for complex networks with a small-world structure, the number of colors required for distance-$d$ colorings is larger, so the probing approach becomes more expensive. For this type of graphs, the randomized method is more competitive than the one based on probing, since it is less affected by the structure of the graph; however, for matrices in which low rank approximation cannot be exploited such as the graph Laplacians that we consider, randomized trace estimators are best suited for computing approximations with a relatively low accuracy, since their cost quickly grows as the requested accuracy is increased.

\section*{Acknowledgements}

We would like to thank the editor Ilse Ipsen and two anonymous reviewers for their insightful comments.

\section*{Declarations}
\subsection*{Funding}
We acknowledge financial support by MUR (Italian Ministry for University and Research) through the PNRR MUR project PE0000023-NQSTI and by INDAM (Italian Institute of High Mathematics) through the INDAM-GNCS project ``Metodi basati su matrici e tensori strutturati per problemi di algebra lineare di grandi dimensioni''.
Funding for the second and third author's PhD scholarship is provided by the
``Departments of Excellence" program of the Italian Ministry for University and Research. 
\subsection*{Competing Interests}
The authors declare no competing interests.

\end{document}